\documentclass[10pt]{article}
\usepackage[margin=1.0in]{geometry}
\usepackage{amsmath,amsthm,amssymb}
\usepackage{tikz}
\usepackage{bm}
\usepackage[utf8x]{inputenc}
\usepackage{color}
\definecolor{grn}{rgb}{0,0.4,0}
\definecolor{dgrn}{rgb}{0.0,0.3,0.0}
\definecolor{dpur}{rgb}{0.3,0.0,0.6}
\usepackage[mathscr]{euscript}
\usepackage[colorlinks=true, linkcolor=dpur]{hyperref}

\newtheorem{theorem}{Theorem}[section]
\newtheorem{proposition}{Proposition}[section]
\newtheorem{corollary}{Corollary}[section]
\newtheorem{lemma}{Lemma}[section]

\theoremstyle{definition}
\newtheorem{remark}{Remark}[section]
\numberwithin{equation}{section}

\newcommand{\mipc}[1]{{\color{blue}[#1]}}

\newcommand{\pomega}{\overrightarrow{\partial\Omega}}
\newcommand{\oo}{\overline\Omega}
\newcommand{\bo}{\overline\Omega\setminus\Omega}
\newcommand{\R}{\mathbb R}

\title{Sharp discrete isoperimetric inequalities in periodic graphs via discrete PDE and Semidiscrete Optimal Transport}

\author{Mircea Petrache\footnote{mpetrache@mat.uc.cl, Facultad de Matem\'aticas, Pontificia Universidad Cat\'olica de Chile, Avda. Vicu\~na Mackenna 4860, Macul, Santiago, 6904441, Chile.} \and Mat\'ias G\'omez\footnote{matias.gomeza.14@sansano.usm.cl}
}

\date{}
\begin{document}
\maketitle

\begin{abstract}
We develop criteria based on a calibration argument via discrete PDE and semidiscrete optimal transport, for finding sharp isoperimetric inequalities of the form $(\sharp \Omega)^{d-1} \le C (\sharp \overrightarrow{\partial\Omega})^d$ where $\Omega$ is a subset of vertices of a graph and $\overrightarrow{\partial\Omega}$ is the oriented edge-boundary of $\Omega$, as well as the optimum isoperimetric shapes $\Omega$. The method is a discrete counterpart to Optimal Transport and ABP method proofs valid in the continuum, and answers a question appearing in Hamamuki \cite{hamamuki}, extending that work valid for rectangular grids, to a larger class of graphs, including graphs dual to simplicial meshes of equal volume. We also connect the problem to the theory Voronoi tessellations and of Aleksandrov solutions from semidiscrete optimal transport. The role of the geometric-arithmetic inequality that was used in previous works in the continuum case and in the $\mathbb Z^d$-graph case is now played by a geometric cell-optimization constant, where the optimization problem is like in Minkowski's proof of his classical theorem for convex polyhedra. Finally, we study the optimal constant in the related discrete Neumann boundary problem, and present a series of possible directions for a further classification of discrete edge-isoperimetric constants and shapes.
\end{abstract}

\section{Introduction}

\subsection{Minimum energy and discrete Wulff shape in a graph}\label{introintro}
In the study of crystals and of other problems, one is lead to study the optimum shape of $N$-point systems at fixed $N$, minimizing an energy given by simple pairwise interactions, and its asymptotic in the limit $N\to\infty$. 

\medskip

As a basic model problem consider the following. If we study the shape of a model of crystallized material, we consider configurations which are subsets of a lattice $V = M \mathbb Z^d\subset \mathbb R^d$ with $M:\mathbb R^d\to\mathbb R^d$ an invertible linear operator. We further imagine that points in $V$ only interact in pairs, and that only $x, y$ such that $x-y\in\mathcal V$ for a fixed finite $\mathcal V\subset \mathbb R^d$, i.e. pairs of particles that are nearest-neighbors in finitely many fixed ``lattice-coordinate'', will be bonded together. If a symmetric $g:V\times V\to [0, +\infty)$ of the form $g(x,y)=\phi(x-y)$ measures how much energy is required to break bonds in the system, then a perfect crystal shape will be the one requiring the most energy to completely break apart, and breaking apart a finite configuration $\Omega\subset V$ would require energy
\begin{equation}\label{minen}
\mathcal E_g(\Omega)=\sum_{\substack{x,y\in\Omega\\x-y\in\mathcal V}} g(x,y) = \sharp \Omega \sum_{v\in\mathcal V} \phi(v) - \sum_{(x,y)\in\overrightarrow{\partial \Omega}}\phi(x-y),
\end{equation}
where 
\[
\overrightarrow{\partial\Omega}:=\{(x,y)\in\Lambda\times\Lambda:\ x\in \Omega, y\notin\Omega, x-y\in\mathcal V\}.
\]
Note that in $V$ each point may require at most a total of $\sum_{v\in \mathcal V} \phi(v)$ to be separated from its neighbors, thus the first term in \eqref{minen} is the ``bulk energy of the configuration" and is independent of the shape of $\Omega$. Then finding the most robust crystal shape of $\Omega\subset\Lambda$ at fixed $\sharp\Omega$ is equivalent to minimizing the "discrete weighted surface area" here represented by 
\begin{equation}\label{gperi}
\sharp_g\overrightarrow{\partial\Omega}:=-\sum_{(x,y)\in\overrightarrow{\partial \Omega}}g(x,y).
\end{equation}
If a minimum-energy $\Omega$ exists at fixed $\sharp \Omega$ and $\sharp_g\overrightarrow{\partial\Omega}$ has growth $O\left((\sharp \Omega)^{\frac{d-1}d}\right)$ like an area term, similarly to the continuum case, then we also have an isoperimetric inequality which has the following scaling-invariant form:
\begin{equation}\label{minen1}
(\sharp \Omega)^{d-1} \le C_{iso} (\sharp_g\overrightarrow{\partial\Omega})^d.
\end{equation}
More generally we work in the following setup:
\begin{itemize}
 \item Let $G=(V,A,E)$ be an ambient graph, with $V\subset \mathbb R^d$ a discrete set, $A:V\times V\to \mathbb R$ a symmetric weight function, and edge set $E=\{(x,y)\in V\times V:\ A(x,y)\neq 0\}$. 
 \item Consider a symmetric function $g:E\to [0, +\infty)$ which represents the cost of edges of $G$.
\end{itemize}
Then we define $\sharp_g\overrightarrow{\partial\Omega}$ as in \eqref{gperi} and we desire to find sharp inequalities of the form \eqref{minen1} and to characterize the shape of configurations $\Omega\subset V$ for wihch equality holds in \eqref{minen1}. The optimum value $C_{iso}$ will be called the sharp isoperimetric constant of the weighted graph $G$.

\subsection{Wulff shape theory in the continuum}
The continuum analogue of the cardinality $\sharp \Omega$ for a domain $\Omega\subset \R^d$ will be its Lebesgue measure $|\Omega|=\mathcal L^d(\Omega)$, and as a replacement for the surface energy \eqref{minen1} we introduce the usual anisotropic perimeter of a set of finite perimeter:
\[
\mathrm{Per}_\phi(\Omega):=\int_{\partial^*\Omega}\phi(\nu(x)) d\mathcal H^{d-1}(x),
\]
where $\phi:\R^d\to\R$ is a convex positive and positively $1$-homogeneous symmetric weight function (i.e. a quasinorm), $\mathcal H^k$ is the $k$-dimensional Hausdorff measure and $\partial^*\Omega$ is the reduced boundary, on which the measure-theoretic normal $\nu:\partial^*\Omega\to\mathbb S^{d-1}$ can be defined (see \cite{maggibook}). Then the Wulff shape $H_\phi$ associated to $\phi$ is the {\bf unit ball of the seminorm dual to $\phi$}, i.e. 
\begin{equation}\label{wulff}
H_\phi:=\{p:\ \forall x\in\R^d,\ p\cdot x < \phi(x)\}.
\end{equation}
This is not the most general setup for the study of Wulff shapes: convexity and symmetry of $\phi$ can be withdrawn, see e.g. \cite{taylor}. For more general isoperimetric problems see \cite{osserman1978isoperimetric, brothers1994isoperimetric, taylor, burago2013geometric}.

\medskip

\subsubsection{Connection between different fields} We focus only two of many possible strategies for \eqref{minen1}: 
\begin{enumerate} \item[(a)] A proof strategy of Gromov \cite{gromov1986isoperimetric}, based on Knothe's transport plans \cite{knothe1957contributions}, then improved based on the Brenier \cite{brenier} Optimal Transport theory (for which see e.g. \cite{santambrogio2015optimal}) by the authors of \cite{cordero2004mass, FMP}.
\item[(b)] The PDE method as in works by Trudinger \cite{trudinger1994isoperimetric}, brought to fruition by Cabr\'e and collaborators in \cite{cabre2000partial, ros2016sharp}, see the survey \cite{cabre2017isoperimetric} for details;
\end{enumerate}
 Both methods give inequalities like \eqref{minen1} in the continuum case, and we describe both below. 

In the discrete case the Cabr\'e setup was generalized by \cite{hamamuki} in the case of orthogonal product lattices, however we find that for the general case a semidiscrete optimal transport interpretation is more instructive. 

The theory of Semidiscrete Optimal Transport which we use here becomes a useful tool in several discretization problems, see \cite{gu2016variational} and the refences therein. The various connections are surveyed based on the work of Aurenhammer \cite{aurenhammer1987criterion} and Rybnikov \cite{rybnikov} in Section \ref{survey}. 

Furthermore, it is worth mentioning that the Wulff shape was studied in the original works by Wulff \cite{wulff} and Herring \cite{herring}, see also \cite{taylor, brothers1994isoperimetric} for a later more general setting extension. 

The discrete-continuum limit of Wulff shapes indicates a connection between the continuum theory and the discrete theory of isoperimetric shapes. This approach is often performed via a $\Gamma$-limit procedure of the underlying energy functionals. For this direction we mention the work \cite{petdel} and references therein. It is worth mentioning that Optimal Transport methods can be used also in that study, as done in \cite{cicalese2019maximal}. Studying the discrete isoperimetric problem without taking the limit of infinitely large configurations is relevant to crystallization problems in which the lattice structure is not fixed, but rather is determined by a minimization problem, such as proved in dimensions $2,3,8,24$ in \cite{theil2006proof, flatley2015face, cohn2017sphere, cohn2019universal, bedepe}.

\subsubsection{The Mass Transportation proof, following strategy (a)} A first way, based on to obtain the isoperimetric inequality is using the basic setup from mass transport theory, for more references see above. We consider a transport plan $T:\Omega\to H_\phi$, such that the pushforward under $T$ of the measure $\mu$ with density $1_\Omega$ to the measure $\nu$ with density $\frac{|\Omega|}{|H_\phi|}1_{H_\phi}$. Then we have 
\begin{equation}\label{1stcontpf}
 d|\Omega|\left(\frac{|H_\phi|}{|\Omega|}\right)^{\frac1d}=d \int_\Omega (\mathrm{det}\nabla T)^{\frac1d} \le \int_\Omega \mathrm{div} T =\int_{\partial^* \Omega} T\cdot \nu_\Omega \ d\mathcal H^{d-1}(x)=\int_{\partial^*\Omega}|T|_{H_\phi} \phi(\nu_\Omega)d\mathcal H^{d-1}(x) \le \mathrm{Per}_\phi(\Omega),
\end{equation}
where we used, in order, the transport condition on $T$, 
then the arithmetic-geometric mean inequality for the eigenvalues of the matrix $\nabla T$, 
the divergence theorem, and then we introduce the quasinorm $|x|_H:=\inf\{\lambda>0:\ x/\lambda \in H\}$, 
and $|T(x)|_{H_\phi}\le 1$ for $x\in\Omega$ because $T(\Omega)\subset H_\phi$ by definition, and the definition of $\mathrm{Per}_\phi(\Omega)$.

\medskip

Note that in general the first inequality in \eqref{1stcontpf} is an equality if and only if $\nabla T$ almost everywhere is a multiple of the identity matrix, and the second inequality of \eqref{1stcontpf} is an equality if and only if $\mathcal H^{d-1}$-almost everywhere on $\partial^*\Omega$ we have $T\cdot\nu_\Omega=|T|_{H_\phi}\phi(\nu_\Omega)$, which implies the characterization of the isoperimetric shape
 $\Omega=H_\phi$, of the optimal $T$, and of the precise constant in \eqref{1stcontpf}.

\medskip

\subsubsection{The PDE approach, following strategy (b)} A second approach is linked to the first one (at least for the case of convex $\Omega$ and monotone $T$) via the Brenier characterization of the (optimal) transport map as subdifferential of convex function, thus to use $T=\nabla u$ as the ``parameter" in our inequalities. This second approach focuses on the linear $\mathrm{div}T=\Delta u$ part of the inequalities, rather than on the nonlinear $\mathrm{det}\nabla T$ part, achieving the same conclusion even for general $\Omega$. We pass to describe this second strategy.

\medskip

We consider solutions $u:\Omega\to\R^d$ of 
\begin{equation}\label{contneuman}
\left\{\begin{array}{ll}\Delta u =\frac{P_\phi(\Omega)}{|\Omega|}& \text{ on }\Omega\\[3mm]
\partial_{\nu_\Omega} u= \phi(\nu_\Omega)&\text{ on }\partial\Omega.\end{array}\right.
\end{equation}
Note that imposing that $\Delta u$ is constant, the value of the constant is determined by the second equation of \eqref{contneuman} via the divergence theorem applied to $\nabla u$. Rather than the previous $T(\Omega)\subset H_\phi$, we now have the condition that $H_\phi\subset \partial u(\Omega)$, where $\partial u(\cdot)$ is the subdifferential measurable mapping associated to $u$:
\[
\partial u(E):=\left\{\begin{array}{ll}\{p\in\R^d:\ \exists x_0\in E,\ \forall x\in\overline\Omega,\ u(x)\ge u(x_0) + p\cdot(x-x_0)\}&\text{ if }u\in C^0(\overline\Omega)\\[3mm]
\{\nabla u(x_0):\ x_0\in E,\ \forall x\in\overline\Omega,\ u(x)\ge u(x_0) + \nabla u(x_0)\cdot(x-x_0)\}&\text{ if }u\in C^1(\overline\Omega).
\end{array}\right.
\]
For the case $u\in C^1$, the condition that with notation \eqref{wulff} we have $H_\phi\subset \partial u(\Omega)$, follows from the boundary condition in \eqref{contneuman}: if $p\in H_\phi$ then either $u(x)-p\cdot x$ achieves its minimum over $\overline \Omega$ at some interior point $x_0\in\Omega$ then $p=\nabla u(x_0)\in \partial u(\Omega)$. If the minimum is only achieved at points $x_0\in\partial\Omega$ then $\partial_{\nu_\Omega}(u(x)-p\cdot x)\le 0$, which together with the boundary datum in \eqref{contneuman} at $x_0$ gives $\phi(\nu_\Omega(x_0))=\partial_{\nu_\Omega}u(x_0)\le p\cdot\nu_\Omega(x_0)$, which implies that $p\notin H_\phi$, a contradiction.

\medskip

Now we can write, with notation $\Gamma_u:=\{x_0\in\Omega:\ \forall x\in\overline\Omega, u(x)\ge u(x_0)+\nabla u(x_0)\cdot(x-x_0)\}$, 
\begin{equation}\label{2ndcontpf}
|H_\phi|\le |\partial u(\overline\Omega)| =\int_{\nabla u(\Gamma_u)}dp \le \int_{\Gamma_u}|\mathrm{det}(\nabla^2 u)| \le \int_{\Gamma_u}\left(\frac{\Delta u}{d}\right)^d \le |\Gamma_u|\left(\frac{d\ P_\phi(\Omega)}{|\Omega|}\right)^d\le\frac{1}{d^d}\frac{[P_\phi(\Omega)]^d}{|\Omega|^{d-1}}.
\end{equation}
Besides the above justified property that $H_\phi\subset \partial u(\Omega)=\nabla u(\Gamma_u)$, we use $T_u:=\nabla u$ as a change of variables, then we use the arithmetic-geometric inequality for the eigenvalues of $\nabla^2 u$. The forelast equality follows from \eqref{contneuman} and the last one follows because $\Gamma_u\subset\Omega$. As above, the inequalities become equalities only if $\Gamma_u=\Omega, \nabla^2u =\lambda \mathrm{Id}$ and $\Omega=\lambda H_\phi$ for some $\lambda>0$.

\subsection{Discrete optimum isoperimetric inequalities via PDE and Optimal Transport, beyond product graphs}
A very large variety of isoperimetric inequalities on graphs are known, valid in varying generality, see e.g. \cite{bollobas1991edge,bezrukov1994isoperimetric} for edge-isoperimetric inequalities in the grid and in other graphs. See  \cite{block1957discrete} \cite{hadwiger1972gitterperiodische}\cite{bokowski1974ungleichung},\cite{schnell1991two}, \cite{brass1999isoperimetric} \cite{hillock2002inequalities} as well and \cite{chung2004discrete} for a general theory and the relation to spectral theory.

However, virtually all versions of the inequalities do not treat the question of determining the isoperimetric shapes, or discussing cases in which equality in the isoperimetric inequality can be reached. The only exception which we are aware of is the paper \cite{hamamuki}, which treats the case of a lattice $V=D\mathbb Z^d$, with $D=\mathrm{diag}(\lambda_1,\ldots,\lambda_n)$ a diagonal matrix, links only between nearest neighbors in the coordinate directions, and edge weights equal to surface areas of a box with sides parallel to the coordinate directions. The setup in that case is based on the Cabr\'e method from \eqref{2ndcontpf}. The maximum principle tools from \cite{hamamuki} can be seen as fitting within the broader theory of discretization of PDE tools, for which we refer to the above work for a broad list of references, and here we only mention the works of Merkov \cite{merkov1985second} and the in-depth treatment by Kuo and Trudinger, see e.g. \cite{kuo1990linear}, \cite{kuo1996positive}, \cite{kuo2000note} and the references therein.

\medskip

In the ``product case" $D\mathbb Z^d$ it is straightforward to obtain the discrete counterpart of the geometric-arithmetic mean inequality for eigenvalues of the Hessian, which in the continuum case reads
\begin{equation}\label{hessineq}
|\mathrm{det}(\nabla^2u)|^{1/d}=|\lambda_1\cdots\lambda_d|^{\frac1d}\le\frac{\lambda_1+\cdots+\lambda_d}{d}= \frac{\Delta u}d.
\end{equation}
Note that if $V=\lambda \mathbb Z$ the absolute value of the discrete Laplacian $|\Delta_{\lambda \mathbb Z}u(x)|=|u(x-\lambda) -2 u(x) + u(x+\lambda)|$ is equal to the size of the ``proximal" subdifferential and we have
\begin{eqnarray*}
\{p\in \mathbb R: \forall y\in \overline\Omega,\ u(y)\ge u(x) + p\cdot (x-y)\}=\partial u(x)&\le&\partial^{\mathrm{prox}}u(x):=\{p\in\mathbb R: u(x\pm\lambda)\ge u(x) \pm \lambda p\}\\
&=&\frac1{\lambda}[u(x) - u(x-\lambda), u(x+\lambda) - u(x)].
\end{eqnarray*}
For $D\mathbb Z^d = \prod_{i=1}^d \lambda_i \mathbb Z$ the nearest-neighbor differential is a product of the above, which allows to use the geometric-arithmetic mean inequality directly. 

\medskip

In \cite{hamamuki} the question of correctly geneneralizing the result to obtain isoperimetric constants in other lattices is mentioned as an interesting open problem at the end of their paper. This is the main motivation and direction of the present work. 

\subsection{The strategy: chain of inequalities \eqref{2ndcontpf} in the discrete case, 
and connection to \eqref{1stcontpf}}\label{summary}

The criteria we find for the existence of a discrete sharp isoperimetric inequality like \eqref{minen1} in geometric graphs, will include the possibility of adding specific edge weights. We will see in Section \ref{s6} that our results include and widely extend the ones from \cite{hamamuki}. On the other hand, we also see that the form \eqref{minen1} is not always the sharpest form of the inequality that one should expect, as exemplified by the case of the triangular graph, in which we get \eqref{mainthmintro} (see Section \ref{trigraph} for the proof). This is an interesting new level of complexity for the discrete isoperimetric problem, compared to the continuum isoperimetric inequality.

\medskip

The correct discrete analogue of the geometric-arithmetic inequality, mentioned as an open question in \cite{hamamuki}, turns out to be closely related to the Minkowski theorem in convex geometry, applied to the proximal subdifferentials of the discrete version of the map $u$. See Section \ref{sec4}. 

\medskip

Another new insight that appears in our generalized geometric framework, is that the whole theory of semidiscrete Optimal Transport comes to help, in a way seemingly related a discrete counterpart to proof strategy \eqref{1stcontpf}. More precisely, the $u$ from \eqref{2ndcontpf} can be interpreted as an Alexandrov weak solution to a Semidiscrete Optimal Transport problem see Theorem \ref{mainthmintro}, see e.g. \cite{benamoufroese} and  \cite{gutierrez2001monge, bakelman2012convex} for the more classical theory of Monge-Amp\'ere equations.

\subsubsection{Discrete chain of inequality paralleling \eqref{2ndcontpf}}
If $\overrightarrow{\partial\Omega}:=\{(x,y):\ x\in\Omega,\ y\notin\Omega,\ A(x,y)\neq 0\}$ and $g:\overrightarrow{\partial \Omega}\to \R$ is fixed, we solve a discrete Neumann value problem formulated as follows, for $\Omega\subset  V$. As noted by \cite{hamamuki}, the continuum equation \eqref{contneuman} must be replaced by the below differential inequality \eqref{neumannintro} in the discrete case. The upper bound on $\Delta_Au$ in \eqref{neumannintro} is fixed as in \cite{hamamuki} so that we are sure that a solution exists, however in some cases it is not optimal. We give a first characterization the sharp optimal constant in Section \ref{sec3}, however this is not explicit enough to help in practice, and further studies are left to future work. Therefore, we use the below bound \eqref{neumannintro} for our study. 
\begin{equation}\label{neumannintro}
\left\{\begin{array}{l}\Delta_A u(x) :=\sum_{y\in V}A^2(y,x)(u(x)- u(y)) \le \frac{1}{\sharp\Omega}\sum_{(x,y)\in\overrightarrow{\partial\Omega}}g(x,y)\quad\text{ for }x\in\Omega\\[3mm]
\frac{\partial u}{\partial_{\nu_\Omega}}=\frac{g}{A}\quad\text{ on } \overrightarrow{\partial\Omega}.
 \end{array}\right.
\end{equation}
The boundary condition in \eqref{neumannintro} has two intepretations: 
\begin{itemize}\item either the na\"ive one, stating that $u(y)-u(x)=\frac{g(x,y)}{A(x,y)}$ for all $(x,y)\in\overrightarrow{\partial\Omega}$, \item or the one set up by Hamamuki in \cite{hamamuki}, requiring that for each $y$ outside $\Omega$ having some neighbor in $\Omega$, there exist such neighbors $x$ such that $u(y)-u(x)=\frac{g(x,y)}{A(x,y)}$.
\end{itemize}
 The na\"ive interpretation will be discarded as it has the drawback that for general nonconvex domains $\Omega$ a solution to \eqref{neumannintro} may not exist, whereas the Hamamuki interpretation allows to find a solution. A more detailed explanation is included in Sections \ref{npnaive} and \ref{nphamamuki}.

\medskip

We assume that $u$ is a solution to \eqref{neumannintro} with Hamamuki interpretation for the boundary condition. To this discrete $u$ we associate subdifferentials $\partial_\Omega u(x)$ depending on the values of $u$ on $\Omega$ and a target shape determined by the boundary values in \eqref{neumannintro}:
\begin{eqnarray}\label{defhg}
\partial_\Omega u(x)&:=&\{ p \in \mathbb{R}^d :\ ( \forall \ z \in \overline{\Omega} ), \  u(x) \leq u(z) + p \cdot (x-z) \},\nonumber\\
\partial^{\mathrm{prox}}u(x)&:=&\{ p \in \mathbb{R}^d :\ ( \forall \ z:\ z\sim x), \  u(x) \leq u(z) + p \cdot (x-z) \},\nonumber\\
H_g&:=&\left\{p\in\mathbb R^d:\ \forall(x,y)\in\overrightarrow{\partial\Omega}\ p\cdot(y-x)\le \frac{g(x,y)}{A(x,y)}\right\}.
\end{eqnarray}
We will also write $\partial u(x)$ rather than $\partial_\Omega u(x)$ if the choice of $\Omega$ is clear from the context.

\medskip

\noindent The discrete analogue of \eqref{2ndcontpf} states that the following holds
\begin{multline}\label{proofscheme}
 |H_g|\stackrel{(a)}{\le} \left|\bigcup_{x\in\Omega}\partial u(x)\right|\stackrel{(b)}{=}\sum_{x\in\Omega}|\partial u(x)|\stackrel{(c)}{\le} \sum_{x\in\Omega}|\partial^{\mathrm{prox}}u(x)| \\
 \stackrel{(d)}{\le} \sum_{x\in \Omega}c_x\left(\Delta_A u(x)\right)^d \stackrel{(e)}{\le} (\max_{x\in\Omega}c_x) \sharp \Omega \left(\frac{1}{\sharp \Omega}\sum_{(x,y)\in\overrightarrow{\partial \Omega}}g(x,y)\right)^d:=\max_{x\in\Omega}c_x\frac{\left(\sharp_g\overrightarrow{\partial\Omega}\right)^d}{(\sharp\Omega)^{d-1}}.
\end{multline}
The above inequalities are proved and discussed as follows:
\begin{enumerate}
 \item[(a):] The boundary conditions of \eqref{neumannintro} ensure that $H_g\subset \bigcup_{x\in\Omega}\partial u(x):=\partial u(\Omega)$, like in the continuous case: fixing an element $p\in H_g$, if $u(x) - p\cdot x$ achieves its minimum over $x\in\overline\Omega$ on $\Omega$ then $p\in \partial u(\Omega)$; if not, then there exists $(x,y)\in\overrightarrow{\partial\Omega}$ such that $u(y)-p\cdot x<u(x)-p\cdot x$: we then have
 \[
 \frac{g(x,y)}{A(x,y)} = u(y) - u(x) < p\cdot (y-x),
 \]
 contradicting the assumption that $p\in H_g$. Furthermore, note that if equality holds in (a) and $u$ is the restriction of a convex function, then $H_g=\partial u(\Omega)$ and $\Omega$ is of the form $\Omega= K\cap V$ with $K\subset\mathbb R^d$ a convex set.
 \item[(b):] The subdifferentials of $u$ are essentially disjoint (i.e. they have zero measure intersection), as we show in Lemma \ref{subdiffdisjoint}. If equality also holds in (a) then the $\{\partial u(x), x\in\Omega\}$, form a partition of $H_g$.
 \item[(c):] In general, the proximal subdifferential $\partial^{\mathrm{prox}}u(x)$ includes $\partial_\Omega u(x)$ but is not equal to it. We find in Proposition \ref{proxsubdiff} a condition on the underlying graph edges $\vec E=\{(x,y):\ A(x,y)\neq 0\}$ which ensure that $\partial^{\mathrm{prox}}u(x)=\partial_\Omega u(x)\neq \emptyset$ holds at all $x\in\Omega$ only if $u$ is convex.
 \item[(d):] This step is a geometric discrete arithmetic-geometric mean inequality, which depends on the neighborhood geometry near $x\in\Omega$. Indeed, the volume of the polyhedron $\partial^{\mathrm{prox}}u(x)$ and the discrete Laplacian $\Delta_A u(x)=\sum_{y:y\sim x}A^2(x,y)(u(x)-u(y))$ coincide with an arithmetic and geometric means in the case of a rectangular lattice treated in \cite{hamamuki}. In general, we have a dilation-invariant bound 
 \begin{equation}\label{subdiffintro}
  |\partial^{\mathrm{prox}} u(x)| \le c_x \left(\Delta_A u(x)\right)^d,
 \end{equation}
 where the optimal constant $c_x=C_{\mathcal V_x}$ depends on the geometry of the neighbor directions $\mathcal V_x:=\{(y-x)/|y-x|:\ y\simeq x\}$ and on $A$, and is related to a version of the Minkowski theorem from convex geometry, described in Section \ref{sec4}. As indicated in Theorem \ref{optsd}, the optimum shape of $\partial^{\mathrm{prox}}u(x)$ so that an equality appears in this step, is unique up to translation: it will depend only on $\mathcal V_x, A$.
 \item[(e):] The quantity $C(\Omega, g)$ in \eqref{neumannintro} obtained formally from the divergence theorem for the operator $\Delta_A$, is not necessarily optimal for general geometries, however allows to solve \eqref{neumannintro} as discussed in \cite{hamamuki} in the product case. This value of $C(\Omega,g)$ is sufficient for sharp bounds in several important examples, as indicated in Section \ref{s6}. We find a characterization of the optimal constant in the general case in Theorem \ref{thmconstnpintro}, however we do not investigate its applications here, as it is based on a not well-understood discrete equation \eqref{veqn} on directed graphs. See section \ref{sec3} for details. Finally, note that if the $G$ is connected and there is an equality in (d), then the $c_x$ are all equal to each other. See Theorem \ref{specialisointro} and Section \ref{sec4} for details.
\end{enumerate}
\subsubsection{Interpretation via Semidiscrete Optimal Transport, following \eqref{1stcontpf}}\label{aleksintro}
We first recall, following \cite{benamoufroese}, the theory of semidiscrete optimal transport solutions in our setting. To imitate the framework \eqref{1stcontpf}, we fix 
\[
  \mu:=\frac{|H_\phi|}{\sharp\Omega}\sum_{x\in\Omega}\delta_x,\quad \nu:=1_{H_\phi}(x) dx.
\]
However, rather than finding a map satisfying $T_\sharp\mu=\nu$ (which is impossible since the pushforward by a measurable map of an atomic measure is always an atomic measure), we look for $T$ going in the opposite direction, and consider an ``optimal" map $T:H_\phi\to \mathbb R^d$ such that $T_\sharp \nu = \mu$, i.e. $T$ is also minimizing the transport cost $\int_{H_\phi}|x- T(x)|^2d\nu(x)$ under the constraint that $T_\sharp\nu=\mu$. In this case the map exists and, due to Brenier's theory \cite{brenier}, it can be represented as $T=\nabla \phi$ for a (unique up to summing constant) convex $\phi:H_\phi\to\mathbb R$. The map $\phi$ is a so-called \emph{Pogorelov solution} to the optimal transport problem between $\nu$ and $\mu$ with cost $|x-y|^2$. An \emph{Aleksandrov solution} to the same problem is a convex function $\phi^*:\mathbb R^d\to \mathbb R$ such that the subdifferential map $\partial \phi^*$ satisfies $|\partial \phi^*(E)|=\mu(E)$ for $E\subset\mathbb R^d$ and $\partial \phi^*(\Omega)=H_\phi$. For the Pogorelov solution there exist $c_x\in \mathbb R, C_x\subset H_\phi$ with $|C_x|=|H_\phi|/\sharp \Omega$ for all $x\in \Omega$, which are related to $\phi$ as follows:
\[
 \forall y\in H_\phi,\ \phi(y)=\max_{x\in\Omega}\{y\cdot x - c_x\}, \quad\quad \forall x\in \Omega,\ C_x=\{y\in H_\phi:\ \nabla\phi(y)=x\}=\{y\in H_\phi:\ \phi(y)=y\cdot x - c_x\}.
\]
The Legendre transform of a Pogorelov solution is a so-called Aleksandrov solution $\phi^*(x):=\sup_{y\in\mathbb R^d}\{x\cdot y - \phi(y)\}$, which satisfies 
\[
\forall x\in\Omega,\quad \phi^*(x)=c_x,\quad\text{and}\quad \partial\phi^*(x)=C_x.
\]
\subsection{Summary of new contributions of this work}
We start by formulating the optimization problem for finite sets $\mathcal V\subset \mathbb R^d$ (repeated in \eqref{subdopt1}), which plays a role in step (d) of \eqref{proofscheme}:
\begin{equation}\label{subdopt1intro}
 C_{\mathcal V}:=\max\left\{\left|\bigcap_{v\in\mathcal V} \{p:\ \langle p,v\rangle \le c_v\}\right|: \vec c=(c_v)_{v\in\mathcal V}\in\R^{\mathcal V},\ \sum_{v\in\mathcal V}c_v=1\right\}.
\end{equation}
We connect this with Minkowski's theorem \cite[Ch. 7]{alexandrov2005convex} and find that the optimal shape is uniquely determined provided that the origin is in the interior of $\mathrm{Conv}(\mathcal V)$ and that $\sum_v\in\mathcal V=0$. The precise statements are given in Lemma \ref{existcond} and Theorem \ref{optsd}. Then this is applied to the precise case of step (d) in \eqref{proofscheme} in Proposition \ref{equalitysubdiff}.

In Section \ref{sec5} we prove the following:
\begin{theorem}[See Theorem \ref{specialiso} for the complete statement]\label{specialisointro}
For $V,A,G$ as above, assume that local convexity condition \eqref{nbdcond} holds at each $x\in V$ and that $G$ is locally finite. Let $\Omega\subset V$ be a finite subset and let $\overline\Omega$ be the closure of $\Omega$ in $G$. Let $u:\overline{\Omega}\to \mathbb R$ be a solution of \eqref{neumannintro}. Then \eqref{proofscheme} holds and we have the following. 
\medskip

\noindent {\bf Necessary conditions for equality.} If equality holds in \eqref{proofscheme} then all $c_x=C_{\mathcal V_x}, x\in\Omega$ are equal, $G|_{\overline\Omega}$ is the dual graph of a face-to-face decomposition of $\partial u(\Omega)$ into convex polyhedra $\partial^{\mathrm{prox}}u(x),x\in\Omega$ of equal volume, and $u$ satisfies $\Delta_Au(x)=\sharp_g\overrightarrow{\partial\Omega}/\sharp\Omega$ in $\Omega$, $u(y)-u(x)=\frac{g(x,y)}{A(x,y)}$ for all $(x,y)\in\overrightarrow{\partial\Omega}$.
\medskip

\noindent {\bf Sufficient conditions for equality.} Let $\Omega\subset V$ be connected within $G$. Equality in the isoperimetric inequality in graph $G$ with weight $A$ is achieved by $\Omega$ if the following geometric conditions are met by $G,A$:
\begin{enumerate}
 \item The complex made of vertices and edges of $G|_{\Omega}\cup\overrightarrow{\partial\Omega}$ (note that vertices in $\overline{\Omega}\setminus\Omega$ are not included) is reciprocal to the collection of $d$- and $(d-1)$-cells of an equal-volume Voronoi tessellation of a convex polyhedron $H$.
 \item If $F_{x,y}$ denotes the $(d-1)$-dimensional facet of the Voronoi tessellation which is dual to edge $(x,y)$ of $G$, then the $\frac{A^2(x,y)|y-x|}{\mathcal H^{d-1}(F_{x,y})}$ takes the same value for all edges of $G$.
\end{enumerate}
 Under the above conditions, functions $u$ achieving equality in \eqref{ineqchain1} are precisely those of the form $u=\lambda u_{\mathrm{Alek}} + \ell$ where $\lambda>0$, $\ell$ is an affine function and $u_{\mathrm{Alek}}$ is the Aleksandrov solution to the optimal transport problem between $\frac{|H|}{\sharp\Omega}\sum_{x\in\Omega}\delta_x$ and $1_H(x)dx$ with cost $|x-y|^2$.
\end{theorem}

\begin{remark}
 Note that in the second part of the theorem, the Neumann boundary datum of $u_{\mathrm{Alek}}$ is not specified, and we do not have a simple characterization that works ingeneral. However in several examples we are able to deduce it from the geometry of $G,A, H$ case by case.
\end{remark}

The full development of consequences of the above theorem is left to future work. We selected the following weaker version of Theorem \ref{specialisointro} as our main tool to construct new examples of isoperimetric inequalities in periodic graphs.

\begin{proposition}[For a complete statement see Proposition \ref{corstrong}]\label{corstrongintro}
 With notations for $V,G,A,g$ be as in the beginning of Theorem \ref{specialisointro}, assume furthermore that:
 \begin{enumerate}
  \item $G$ is reciprocal to a triangulation of $\mathbb R^d$ by equal volume simplices.
  \item There exist constants $C_1, C_2>0$ such that 
  \begin{equation}\label{conststufintro}
  \frac{\mathcal H^{d-1}(F_{x,y})}{|x-y|A^2(x,y)}=C_1\quad \text{and} \quad \frac{\mathcal H^{d-1}(F_{x,y})|x-y|}{g^2(x,y)}=C_2
  \end{equation}
 for all edges $(x,y)$ of $G$ and all corresponding dual $(d-1)$-facets of the tessellation.
 \end{enumerate}
 Then if $H_g$ is as in \eqref{defhg}, for all $\Omega\subset V$ we have 
\begin{equation}\label{isopineqallintro}
(\sharp \Omega)^{d-1}\ \le\ \frac{C_{\mathcal V}}{|H_g|}\ (\sharp_g\overrightarrow{\partial \Omega})^d,
\end{equation}
with equality if and only if a multiple of $H_g$ is tessellated by simplices dual to points in $X$.
\end{proposition}
As an illustration of possible applications, this corollary is then applied in Section \ref{s6} to find isoperimetric shapes in a variety of lattices beyond the product case from \cite{hamamuki}, including the honeycomb lattices and their affine deformations with various weights, the $1$-skeleton of the Voronoi cells of the BCC lattice, as well as products thereof. A whole class of graphs that fit our criteria are the dual graphs of Coxeter triangulations. A classification of the full range of applicability of the above criteria is left for future work.

\medskip

An example of particular interest is that of the triangular graph, i.e. the $1$-skeleton of a tessellation of $\mathbb R^2$ by equilateral triangles. If we take this graph with equal weights, the criteria of Proposition \ref{corstrongintro} or Theorem \ref{specialisointro} do not apply. However, they do apply in the honeycomb graph dual to it, and in Section \ref{trigraph} we use this duality and a counting argument in order to find the following isoperimetric inequality, valid for all finite subset $\Omega$ in the triangular graph (see \eqref{mainthm}):
\begin{equation}\label{mainthmintro}
\frac{(\sharp\overrightarrow{\partial\Omega}-6)^2}{4\sharp\Omega - \sharp\overrightarrow{\partial\Omega} +2}\ge 12.
\end{equation}
Note that this inequality is more complicated than \eqref{isopineqallintro}, and is sharp as equality is reached by regular hexagonal shapes. This shows that in the discrete case inequalities like \eqref{minen1} may not be sharp even in very regular cases, making the study of sharp discrete isoperimetric inequalities mathematically richer than the continuum case (note that in the limit $1\ll\sharp\overrightarrow{\partial\Omega}\ll\sharp\Omega$, we get that \eqref{mainthmintro} recovers a form close to \eqref{minen1}).

\medskip

A possible direction of improvement of the result of Theorem \ref{specialisointro} consists in studying equation \eqref{neumannintro} in further detail, and determining further properties of the optimum Laplacian upper bound which we can afford. This direction is pursued in Section \ref{sec3}. Interpreting the constraint $\partial u/ \partial \nu_\Omega=g$ in the Hamamuki sense mentioned after \eqref{neumannintro}, we have the following definition of the optimal constant
\begin{equation}\label{minmaxintro}
 C(g,\Omega):=\inf\left\{\max_{x\in\Omega}\Delta u(x):\ u:\oo\to\mathbb R,\ \frac{\partial u}{\partial\nu_\Omega}=g\right\}.
\end{equation}
Due to \cite{hamamuki}, we have the upper bound $C(g,\Omega)\le \frac{\sharp_g\overrightarrow{\partial\Omega}}{\sharp\Omega}$ (see notation \eqref{proofscheme} for $\sharp_g$), and we complement this by a lower bound in Corollary \ref{cor:lbcgomega}. Finally, in \eqref{cgomegal2}, we find the following:
\begin{theorem}\label{thmconstnpintro}
Let $\mathcal L'$ be the class of directed graphs 
 \[
 C(g,\Omega)=\min\left\{\frac{\sum_{y\in\bo}g(y)v_{\vec L'}(y)}{\sum_{x\in\Omega}v_{\vec L'}(x)}:\ \vec L'\in\mathcal L'\right\},
\]
where $v_{\vec L'}$ satisfies the dual equation \eqref{veqn1} on the graph $\vec L'$.
\end{theorem}

\subsection{Structure of the paper}
We introduce basic results on weighted graphs and discrete PDEs in Section \ref{sec2}. Section \ref{sec4} describes the optimal subdifferential shape problem \eqref{subdiffintro}. Section \ref{sec5} includes the proof of Theorem \ref{specialisointro} and Proposition \ref{corstrongintro}. This section also includes a small survey on results which are helpful for reformulating our setup within the theory of weighted Voronoi tessellations, liftings and reciprocal graphs. We however focus on the formulation in terms of reciprocal graphs and semidiscrete optimal transport. Section \ref{s6} provides several explicit examples, beyond the ``product case'', which means here the case of a lattice of the form $D\mathbb Z^d$ in which $D$ is a diagonal matrix. Note that this case, treated in \cite{hamamuki}, is a product of $d$ rescaled copies of $\mathbb Z$. Section \ref{sec3} is devoted to the proof of Theorem \eqref{thmconstnpintro}.

\medskip

{\bf Acknowledgements: } MP is supported by the Chilean Fondecyt Iniciaci\'on grant number 11170264 entitled "Sharp asymptotics for large particle systems and topological singularities".

\section{Preliminaries for the discrete case}\label{sec2}

\subsection{Some useful convex analysis lemmas}

\begin{lemma}\label{lem:difincl} Let $d\ge 1$ be an integer and let $G=(V,L)$ be a graph where $V\subset\mathbb R^d$. Suppose that $\Omega\subset V$ is finite. For any function $u:\overline{\Omega} \to \mathbb{R}$. For all $x\in \Omega$ there holds
\[
\displaystyle \partial_\Omega u(x) \subset \bigcap_{v:\{x,x+v\}\in L} \{p \in \mathbb{R}^d : p \cdot v \leq u(x+v) - u(x)\}.
\]
\end{lemma}
\begin{proof}
We assume that $\partial_\Omega^-u(x)\neq\emptyset$, else there is nothing to prove. Let $p \in \partial_\Omega u(x)$, which means that for all $z \in \overline{\Omega}$ we have that 
\[
  u(x) \leq u(z)- p \cdot (z-x).
\]
If we test the above for $z=x+v$ with $v$ such that $\{x,v\}\in L$, we have that $z \in \overline{\Omega}$, by the definition of $\oo$, thus the above inequality holds for these choices, and it reads
\[
u(x + v) \geq  p \cdot v  + u(x).
\]
We thus have, reordening the inequality, for each $v$ as above,
\begin{equation}\label{ubd}
 p \cdot v \leq u(x + v) - u(x)   .
\end{equation}
\end{proof}

\begin{lemma}\label{subdiffdisjoint}
Assume that $\oo\subset \R^d$ is a closed set and for $u:\oo\to\R$ and $x\in\oo$ define 
\[
 \partial_{\oo}u(x):=\{p\in\mathbb R^d:\ (\forall y\in\oo),\ \langle p, y-x\rangle \le u(y)-u(x)\}.
\]
Then for $x\neq z\in\oo$ the sets $\partial_{\oo}u(x),\partial_{\oo}u(z)$ intersect at most their boundaries.
\end{lemma}
\begin{proof}
The condition $p\in\partial_{\oo}u(x)\cap\partial_{\oo}u(z)$ directly translates into 
\[
(\forall y\in\oo),\ \langle p, y-x\rangle \le u(y)-u(x) \quad\mbox{and}\quad \langle p, y-z\rangle \le u(y)-u(z).
\]
In particular, we test the first condition with $y=z$ and the second with $y=x$ and we find
\[
 \langle p,x-z\rangle = u(x)-u(z).
\]
This implies that $p$ belongs to the boundaries of the two subdifferentials, because the equality in the inequality from the definition of the subdifferentials is achieved at $p$.
\end{proof}

\begin{proposition}\label{proxsubdiff}
Let $G$ be a graph with vertices a discrete set $V\subset\mathbb R^d$ and edges $(x,y)$ such that for all $x\in V$, denoting $y\sim x$ the neighborhood relationship in $G$, there holds 
\begin{equation}\label{nbdcond}
\mathrm{conv}\left(\{y\in V:\ x\sim y\}\right)\cap V = \{x\}\cup\{y\in V:\ x\sim y\}.
\end{equation}
If $u:\overline\Omega\to\mathbb R$ satisfies $\partial^{\mathrm{prox}}u(x)=\partial_\Omega u(x)\neq \emptyset$ for all $x\in\Omega$, then $u$ is the restriction of a convex function $\tilde u:\mathbb R^d\to \mathbb R$.
\end{proposition}
\begin{proof}
We first note that $u$ is the restriction of a convex $\tilde u$ if and only if $\tilde u^\Omega|_{\overline\Omega}=u$, where the convex envelope $\tilde u^\Omega:\mathbb R^d\to \mathbb R$ is defined via the condition 
\[
 \mathrm{epi}(\tilde u^\Omega)=\bigcap\{H_\pi\subset\mathbb R^{d+1}:\ \pi\subset\mathbb R^{d+1}\text{ is a hyperplane such that } \pi\cap \mathrm{epi}(u)=\emptyset\},
\]
where for $f:A\to\mathbb R$ and $\pi=\{(x,t):\ t=\langle p,x\rangle\}\subset\mathbb R^{d+1}$ a non-vertical hyperplane, we define
\[
\mathrm{epi}f :=\{(x,t):\ x\in\mathrm{A},\ t\in\mathbb R, \ t\ge \tilde u(x)\},\quad H_\pi:=\{(x,t):\ t\ge \langle p,x\rangle\}.
\]
We call a plane $\pi$ as above such that $\partial H_\pi\cap \partial\mathrm{epi}(\tilde u^\Omega)$ has dimension $d$ a \emph{nondegenerate support plane of $\mathrm{epi}(\tilde u^{\Omega})$}.

\medskip
Now assume that $G$, $u$ are as in the proposition but that $u\neq \tilde u^\Omega$, which as seen above is equivalent to saying that $u$ is not the restriction of a convex function. Then there exist $x_0\sim y_0\in\Omega$ such that $u(x_0)=\tilde u^\Omega(x_0)$ and $u(y_0)\neq \tilde u^\Omega(y_0)$.

\medskip

We note that $\partial\, \mathrm{epi}(\tilde u^{\Omega})$ is a polytope with vertices of the form $(x,u(x))$ with $x\in \Omega$ and $u(x)=\tilde u^\Omega(x)$, and faces contained in its nondegenerate support planes. Up to reducing to an equivalent problem in lower dimension, and using condition \eqref{nbdcond}, we may assume that there are $x_1,\dots,x_d\in \Omega$ such that $u(x_i)=\tilde u^\Omega(x_i), i=1,\dots,d$ 
and $y_0\in\mathrm{int}(\mathrm{conv}\{x_0,\dots,x_d\})$. 
Then we have 
\begin{equation}\label{kinbdry1}
K:=\mathrm{conv}\{(x_0,u(x_0)), \dots,(x_d,u(x_d))\}\subset\partial\, \mathrm{epi}(\tilde u^\Omega).
\end{equation} 
On the other hand, note that for each $x\in\Omega$ there holds 
\begin{equation}\label{kinbdry2}\mathrm{epi}(\tilde u^\Omega)\subset \bigcap\{H_\pi:\ \pi=\{(x,t): t=u(x_0)+\langle p,(x-x_0)\rangle\}, p\in\partial_\Omega u(x)\},
\end{equation}
and thus due to the hypotheses that $\partial^{\mathrm{prox}}u(x)=\partial_\Omega u(x)$ for all $x\in\Omega$, we find
\begin{equation}\label{kinbdry3}
K\nsubseteq\bigcap\{H_\pi:\ \pi=\{(x,t): t=u(x_0)+\langle p,(x-y_0)\rangle\},\ p\in\partial^{\mathrm{prox}}u(y_0)\}.
\end{equation}
Now \eqref{kinbdry1}, \eqref{kinbdry2} and \eqref{kinbdry3} give $K\nsubseteq K$ which is a contradiction, as desired.

\end{proof}
\begin{remark}
 In dimension $d\ge 2$ it is interesting to prove or disprove whether the condition $\forall x\in\Omega, \partial^{\mathrm{prox}}u(x)=\partial_\Omega u(x)\neq \emptyset$ can be supplemented by a simple condition on $G$ to obtain a condition necessary for $u$ to be the restriction to $V$ of a convex function. To see the difficulty, consider the case of $\Omega\subset \mathbb R^2$ composed by $7$ points, namely the vertices $a_1,\dots,a_6$ of a regular hexagon (indices in clockwise order along the perimeter) and its center $a_0$. Let $u_1(x)=0$ for $x\in\{a_0, a_1,a_3,a_5\}$ and $u_1=1$ otherwise, and let $u_2(x)=0$ for $x\in\{a_0,a_2,a_4,a_6\}$ and $u_2=1$ otherwise. Then for neither choice of $a_0$'s $G$-neighbors $a_1,a_3,a_5$ or $a_2,a_4,a_6$ does it happen that $\partial^{\mathrm{prox}}u_i(a_0)=\partial_\Omega u_i(a_0), i=1,2$, but $u_1,u_2$ are restrictions of convex functions to $\Omega$. We do not pursue the question of what conditions on $G$ allow to replace the conditions from Proposition \ref{proxsubdiff} and obtain a necessary and sufficient condition for convexity of $u$.
\end{remark}

\subsection{Na\"ive discrete Neumann boundary value condition}\label{npnaive} This subsection is includes two well-known elementary results, which we include as a motivation for the next subsection, and for future reference. 

\medskip

Let $G,A$ be as in the introduction and consider functions $f:\Omega\to\mathbb R$, $g:\pomega\to\mathbb R$ and $u:\oo\to\mathbb R$. We are going to define what we mean when we say that $u$ is a ``na\"ive'' discrete solution of the Neumann boundary value problem
\begin{equation}\label{neum1}
\displaystyle  \left\{\begin{array}{rl}
\Delta_A u = f  &  \text{in } \Omega,\\
\frac{\partial u}{\partial \nu_\Omega} = \frac{g}{A} &  \text{on } \overrightarrow{\partial \Omega}.
\end{array}\right. 
\end{equation}
Our na\"ive choice is that the first equation in \eqref{neum1} means that $\Delta_A u(x)=f(x)$ for all $x\in\Omega$ and the second equation will mean that for all $(x,y)\in\pomega$, there holds $u(y)-u(x)=\frac{g(x,y)}{A(x,y)}$. The divergence theorem applied to the combinatorial vector field $\nabla_A u(x,y)=A(x,y)(u(y)-u(x))$, implies that a necessary condition for a solution of \eqref{neum1} to exist is that 
\begin{equation}\label{necneum1}
 \sum_{x\in\Omega}f(x)=\sum_{(x,y)\in\pomega}g(x,y).
\end{equation}
Unfortunately, condition \eqref{necneum1} is not sufficient, unlike what happens in the continuous case. This follows by comparing degrees of freedom of the two sides of \eqref{neum1}: $u$ is determined by the choice of its $\sharp\oo$ independent values over its domain, whereas $f$ and $g$ are similarly determined by $\sharp\Omega +\sharp\partial\Omega$ independent values. We can prove the following simple and very general result
\begin{lemma}\label{naiveneumbd}
If $G=(V,E)$ is a graph with vertices $V$ and unoriented edges $E$. For any finite $\Omega\subset V$, defining the oriented boundary $\overrightarrow{\partial \Omega}:=\{(x,y):\ \{x,y\}\in E,\ x\in\Omega, y\notin\Omega\}$ and $\overline \Omega:=\Omega\cup\{x\in V:\ \exists y\in \Omega \mbox{ such that }\{x,y\}\in E\}$, there holds
\begin{equation}\label{problem}
\sharp(\oo\setminus\Omega)\le \sharp\pomega,
\end{equation}
with equality if and only if for each $y\in \oo\setminus\Omega$ there exists exactly one $x\in\Omega$ such that $\{x,y\}\in E_\Omega$.
\end{lemma}
\begin{proof}
To prove \eqref{problem}, we show that the map $\pomega\ni (x,y)\mapsto y\in\oo\setminus\Omega$ is a surjection. Indeed, by definition any $y\in\oo\setminus\Omega$ is an extreme of an edge $(x,y)$ with the $x\in\Omega$, thus $y$ is in the image of the above map. The map is bijective precisely if such $x$ is unique, completing the proof.
\end{proof}
We consider first the case in which equality holds in \eqref{problem}. In this case we have the following result:

\begin{lemma}\label{naivegood}
 Assume that $\sharp\pomega=\sharp(\oo\setminus\Omega)$. Then for any data $f:\Omega\to\mathbb R$ and $g:\pomega\to\mathbb R$ such that condition \eqref{necneum1} holds on every connected component of $\Omega$, there exists exactly one solution to \eqref{neum1}, up to addition of a function which is constant on each connected component of $\oo$.
\end{lemma}
\begin{proof}
As $\sharp \overrightarrow{\partial\Omega}=\sharp (\oo\setminus\Omega)$, for each $x\in\oo\setminus\Omega$ there exists a unique $y_x\in\Omega$ such that $(y_x,x)\in\overrightarrow{\partial\Omega}$. With this notation we define
\[
L:\{u:\oo\to\mathbb R\}\to\{u:\oo\to\mathbb R\},\quad (Lu)(x):=\left\{
\begin{array}{ll}
-\Delta_A u(x)&\mbox{ if }x\in\Omega,\\[3mm]
A^2(x,y_x)(u(y_x)-u(x))&\mbox{ if }x\in\oo\setminus\Omega.\end{array}\right.
\]
We first show that $L$ is self adjoint with respect to the scalar product given by 
\[
 \langle u,v\rangle:=\sum_{x\in\Omega}u(x)v(x)+\sum_{x\in\oo\setminus\Omega}u(x)v(x).
\]
Indeed we have, denoting $d_x:=\sum_{y\in \oo}A^2(x,y)$ the total squared weight of neighbors of a vertex $x$ in $\oo$,
\begin{eqnarray*}
 \lefteqn{\langle Lu,v\rangle= -\sum_{x\in\Omega}\sum_{y:\{x,y\}\in E}A^2(x,y)(u(y)-u(x))v(x) + \sum_{(x,y)\in\pomega}A^2(x,y)(u(y)-u(x))v(y)}\\
 &=&\sum_{x\in\Omega}d_x\ u(x)v(x)-\sum_{x\in\Omega}\sum_{y\in\oo:\{x,y\}\in E}A^2(x,y)u(y)v(x) +\sum_{x\in\oo\setminus\Omega}A^2(y_x,x)u(x)v(x) -\sum_{x\in\oo\setminus\Omega}\sum_{y\in\oo:\{x,y\}\in E}A^2(x,y)u(x)v(y)\\
 &=&\sum_{x\in\oo}d_x\ u(x)v(x) - \sum_{x\in\oo}\sum_{y\in\oo:\{x,y\}\in E}A^2(x,y)u(y)v(x).
\end{eqnarray*}
The last sum is completely symmetric in $u,v$, because $(x,y)\in \vec E$ if and only if $(y,x)\in \vec E$ as well. This proves that $\langle Lu,v\rangle =\langle v,Lu\rangle$, as desired.

\medskip

Next, we show that a function is in the kernel of $L$ if and only if it is constant on every connected component of $\oo$. Otherwise, suppose that $Lu=0$ and that $u$ has a strict local maximum at $y\in\overline \Omega$. If $y\in\Omega$ then the condition $\Delta_A u(y)=0$ implies that the value $u(y)$ is the $A^2$-weighted average of the values of $u$ at the neighbors of $y$ in $\oo$, which contradicts the fact that $y$ is a strict local maximum. If $x\in\oo\setminus\Omega$ then $u(y_x)=u(x)$ and by our hypothesis $y_x$ is the only neighbor of $x$, again contradicting the local maximum property. This shows that $Lu=0$ implies that $u$ is constant on each connected component of $\oo$.

\medskip

Let $\bar U_j$ with $1\le j\le k$ be an enumeration of the connected components of $\oo$. Then there exists a partition of $\Omega$ given by $U_j, 1\le j\le k$ such that $\bar U_j$ is really the closure of $U_j$, as the notation indicates. As the sets $\bar U_j$ form a partition of $\oo$, considering the equation $Lu(x)=b(x)$ for all $x \in \oo$ translates into requiring $Lu(x)=b(x)$ for $x \in \bar U_j$ and $1\le j\le k$. For fixed $j$, we have $\sharp\overrightarrow{\partial U_j}=\sharp((\oo\setminus\Omega)\cap U_j)$. Thus we can replicate the procedure from the previous paragraph to show that for $1\le j\le k$ the restriction $u|_{\bar U_j}\mapsto Lu|_{\bar U_j}$ is selfadjoint. Therefore the system given by $Lu(x)=b(x)$ for all $x \in \bar U_j$ has a solution if and only if $b \in \langle \vec{1}_{\sharp\overrightarrow{\partial U_j}} \rangle^\perp$. Without loss of generality, we can consider an ordered version of $b$ in which the first coordinates are the values of $f$ at $U_j$ and the final coordinates are the values of $g$ on the edges of $\overrightarrow{\partial U_j}$. Then the fact that $Lu=b$ has a solution is equivalent to
\[
\sum_{x \in U_j} f(x) = \sum_{(x,y) \in \overrightarrow{\partial U_j}}g(x,y)
\]
As $L$ is selfadjoint, the image of $L$ is orthogonal to its kernel, which shows that \eqref{necneum1} is verified on each connected component of $\oo$.
\end{proof}

\begin{remark}[Need of a different boundary value definition] In many cases of interest, such as the triangular graph, or the case of non-convex domains $\Omega$, one can check that not only strict inequality in \eqref{problem} holds, but also, the discrepancy between the two sides is considerable, so that the ``extra degree of freedom'' allowed via equation \eqref{necneum1} is not sufficient to ensure existence of solutions. This means that we need to consider an alternative to our na\"ive discrete version of Neumann boundary value conditions.
\end{remark}

\subsection{A discrete Neumann boundary value condition adapted to subdifferential compraisons (following \cite{hamamuki})}\label{nphamamuki} The discussion from the previous subsection indicates that we should instead consider a different setup for \eqref{neum1}, in which the discrete Neumann boundary value condition is fixed via a function $g:\oo\setminus\Omega\to \mathbb R$. This allows to directly have an equal number of degrees of freedom for $u$ as for the pair $(f,g)$. Then we interpret the condition $\dfrac{\partial u}{\partial \nu_\Omega} = \frac{g}{A}$ from \eqref{neum1} to mean that 
\begin{equation}\label{newbdcond}
(\forall y\in\oo\setminus\Omega)(\exists(x,y)\in\pomega)\quad\mbox{such that}\quad u(y)-u(x)=\frac{g(x,y)}{A(x,y)}.
\end{equation}
This condition is useful because it the minimum requirement that still permits completing the reasoning by contraddiction in step (a) of \eqref{proofscheme}, and prove that $H_g\subset \partial_\Omega u(\Omega)$ for solutions of \eqref{neum1}.

\begin{lemma}\label{lemexistsolineq}
Assume $\Omega\subset V$ is a subset of the vertices of a graph $\overline G=(V,\overline E)$, and let $\oo$ be the extremes of all edges in $\overline E$ which have at least one extreme in $\Omega$. Let $\pomega$ be the oriented edges $(x,y)$ such that $x\in\Omega,y\in\oo\setminus\Omega$. Let $G$ be the new graph which has vertex set $\Omega_G:=\Omega\cup\{x_y:\ (x,y)\in\overrightarrow{\partial\Omega}\}$, and edge set
\[
E_{\Omega_G}:=\bigg\{\{x,y\}:\ x,y\in\Omega\bigg\}\cup\bigg\{\{x,x_y\}:\ (x,y)\in\overrightarrow{\partial \Omega}\bigg\}.
\]
We denote $\partial_G$ the boundary taken in the new graph $G$, and given a function $g:\oo\setminus\Omega\to\mathbb R$, we associate to it $\bar g:\overrightarrow{\partial_G\Omega}\to\mathbb R$ defined by
\[
(\forall y\in\oo\setminus\Omega), (\forall x\in \Omega: (x,y)\in \vec E), \quad \bar g(x,x_y):=g(y).
\]
If $f:\Omega\to\mathbb R$ and $g,\bar g$ as above satisfy on each connected component $\Omega_j$ of $\Omega\simeq \Omega_G$
\begin{equation}\label{necneum11}
 \sum_{x\in\Omega_j}f(x)=\sum_{(x,(x,y))\in\overrightarrow{\partial_G\Omega_j}}\bar g(x,(x,y))=\sum_{y\in\oo\setminus\Omega}\sharp\{x\in\Omega_j:\ \{x,y\}\in \overline E\} g(y),
\end{equation}
then there exists a solution $u_G$ to
\begin{equation}\label{neum11}
\displaystyle  \left\{\begin{array}{rl}
\Delta_A u_G(x)= f(x)  &  x\in \Omega=\Omega_G,\\
u_G(x_y)-u_G(x)= \frac{\bar g(x,x_y)}{\bar A(x,x_y)} &  (x,x_y)\in\overrightarrow{\partial \Omega_G},
\end{array}\right. 
\end{equation}
in which $\Omega_G$ are the vertices in $G$ which are identified with vertices in $\Omega$.

\smallskip

\noindent Moreover, for $f,g$ satisfy \eqref{necneum11}, in the original graph $\overline G$ there exists a solution $u$, which is equal to $u_G$ over $\Omega$, to the system
\begin{equation}\label{neum12}
\displaystyle  \left\{\begin{array}{rl}
\Delta u \le f  &  \text{in } \Omega,\\
\frac{\partial u}{\partial \nu_\Omega} = g &  \text{on } \overrightarrow{\partial \Omega},
\end{array}\right. 
\end{equation}
where the boundary condition is defined by \eqref{newbdcond}.
\end{lemma}
\begin{proof}
From the definition of $\Omega_G$ we have that $\sharp\overrightarrow{\partial_G\Omega}=\sharp(\oo\setminus\Omega)$. The hypothesis that $f:\Omega\to\mathbb R$ and $\bar g:\overrightarrow{\partial_G\Omega}\to\mathbb R$ satisfy 
\eqref{necneum11} means that $f$ and $\overline{g}$ satisfy the condition $\eqref{necneum1}$ on every connected component of $\Omega\simeq\Omega_G$. Then, applying Lemma \ref{naivegood}, there exists exactly one solution to $\eqref{neum11}$, up to addition of a function which is constant on each connected component of $\overline{\Omega}$.

\medskip

If we now define $u: \overline{\Omega} \to \mathbb{R}$ by $u(x)=u_G(x)$ for $x \in \Omega$ and $u(y)=\displaystyle \max_{y:(x,y) \in \overline{E}} u_G(y)$ for $y \in (\overline{\Omega} \backslash \Omega)$ and we take into count the relation between $\overline{g}$ and $g$, we conclude that $u$ satisfies $\eqref{neum12}$.
\end{proof}

\section{An optimization result for polyhedra in $\R^d$ related to Minkowski's theorem}\label{sec4}
We introduce a finite-dimensional optimization problem which it is useful to study for the final discussion. To formulate the problem in its general form, we use the following notation for a halfspace with normal vector $v\in\mathbb R^n$ and boundary the hyperplane $\{p:\ p\cdot v =c\}$:
\[
H_v(c):= \{p: \langle p,v\rangle \le c\}.
\]
where $c\in \mathbb{R}$ is a constant.

\medskip

If we fix $\mathcal V\subset\mathbb R^d$ appropriately, there holds
\begin{equation}\label{arithmgeomgen}
\left|\bigcap_{v\in \mathcal V}H_v(c_v)\right|\le C_{\mathcal V}\left(\sum_{v\in\mathcal V} c_v\right)^d,
\end{equation}
which for $\mathcal V=\{\pm e_1,\ldots,\pm e_d\}\subset\mathbb R^d$ is equivalent to the arithmetic-geometric inequality, for which we have optimal values $c_v=\frac{1}{2d}$ for all $v\in\mathcal V$ and $C_{\mathcal V}=\frac{1}{d^d}$. Thus \eqref{arithmgeomgen} generalizes this inequality.

\medskip

The optimal constant in \eqref{arithmgeomgen} can be expressed as follows:
\begin{equation}\label{subdopt1}
 C_{\mathcal V}=\max\left\{\left|\bigcap_{v\in\mathcal V} H_v(c_v)\right|: \vec c=(c_v)_{v\in\mathcal V}\in\R^{\mathcal V},\ \sum_{v\in\mathcal V}c_v=1\right\}.
\end{equation}
In the above, $|\cdot|$ represents the Lebesgue measure in $\R^d$ and the maximum exists only if the convex hull $\mathrm{conv}(\mathcal V)$ has the origin as an interior point. We will only study this problem for {\bf finite $\mathcal V$}, although it can be studied also in higher generality. We will denote from now on
\begin{equation}\label{notmcv}
 \mathcal V=\{v_1,\ldots,v_N\}.
\end{equation}
The below results turned out to be equivalent to the proof of Minkowski's theorem which describes convex polyhedra with given facet normals and areas, see \cite[Ch. 7]{alexandrov2005convex}. Minkowski introduced a problem equivalent to \eqref{subdopt1} as a tool for constructing polyhedra and showing that they cannot be constructed in some cases. In our case we follow the opposite direction: we start from \eqref{subdopt1} and are interested in the characterization of the solutions. We will state the theorems here, and we sketch the proofs in the appendix. The interested reader can find a large overlap with the above cited book of Alexandrov.
\begin{lemma}\label{existcond}
With the above notation the following hold:
\begin{enumerate}
\item The problem \eqref{subdopt1} has a finite maximum value if and only if 
 \begin{equation}\label{sumzero}
  \mathrm{Span}\mathcal V = \mathbb R^d\quad \text{and}\quad\sum_{v\in\mathcal V} v=0.
 \end{equation}
\item If \eqref{sumzero} holds and $\mathcal R_{\vec c}:=\bigcap_{v\in\mathcal V}H_v(c_v)$ is an optimizer, then $\mathcal R_{\vec c} + w$ is also an optimizer for each $w\in\R^d$.
\end{enumerate}
\end{lemma}
We also provide a proof of the above lemma, for completeness.
\begin{proof}
 We start by proving the second statement of the theorem. It suffices to show that $\mathcal R_{\vec c}+w$ is also a solution to \eqref{subdopt1} if $\mathcal R_{\vec c}$ is. For this, note that 
 \begin{equation}\label{thvc}
  w+H_v(c_v) = H_v(c_v+\langle w, v\rangle).
 \end{equation}
Therefore $\mathcal R_{\vec c} + w=\mathcal R_{\vec c'}$, where $c'_v=c_v+\langle w,v\rangle$ for $v\in\mathcal V$. In particular 
\[
 \sum_{v\in\mathcal V} c'_v=1\ + \ \left\langle w, \sum_{v\in\mathcal V}v\right\rangle=0,
\]
which shows that $\mathcal R_{\vec c}+w$ is also a solution to \eqref{subdopt1} if $\mathcal R_{\vec c}$ is, as desired.

\medskip

To prove the necessity of \eqref{sumzero} for the existence of a maximizer in \eqref{subdopt1}, assume that \eqref{sumzero} does not hold, and let 
\[
 \bar v:=\sum_{v\in\mathcal V}v\neq 0.
\]
Then we use again \eqref{thvc} and for $w=-\lambda\bar v$ if $\vec c\in\R^{\mathcal V}$ satisfies $\sum_{v\in\mathcal V}c_v=1$ and $|\mathcal R_{\vec c}|>0$ then we find 
\[
 w+\mathcal R_{\vec c} = \mathcal R_{\vec c'}:=\mathcal R_{\vec c -\lambda|\bar v|^2\vec 1}, 
\]
in particular if for a fixed $\epsilon\in (0,1)$ we choose $\lambda=\frac{1-\epsilon}{|\bar v|^2\sharp \mathcal V}$ then $\sum_{v\in\mathcal V}c_v'=\epsilon$, and thus we have found $w=w_\epsilon$ as above, such that
\begin{equation}\label{dilatrprime}
 \frac1{\epsilon}\left(\mathcal R_{\vec c} + w_\epsilon\right)\quad\mbox{ is a competitor in \eqref{subdopt1}}.
\end{equation}
But due to the dilation by $1/\epsilon$, the competitor \eqref{dilatrprime} has volume $1/\epsilon^d$ times larger than the one of $\mathcal R_{\vec c}$. This construction can be done for $\epsilon$ arbitrarily close to $0$, which shows that \eqref{subdopt1} has no solution if \eqref{sumzero} fails.

\medskip

It remains to prove that if \eqref{sumzero} holds then a solution to \eqref{subdopt1} exists. First note that the vectors $v\in\mathcal V$ have $0$ in their convex hull, and thus taking all the $c_v\ge 0$ (for example $c_v=1/\sharp \mathcal V$ for all $v\in\mathcal V$), we find $\mathcal R_{\vec c}$ with nonempty interior and contained in the ball centered at zero with radius $\max_{v\in\mathcal V}|v|$, and thus it has bounded volume.

\medskip

A consequence of the above is that competitors with nonempty interior exist, and they are the only ones with nonzero volume because we are in a space of finite dimension. Now by the result of the already proved item 2 of the lemma, under condition \eqref{sumzero} any competitor of nonempty interior to \eqref{subdopt1} can be translated so that the origin is in its interior, in which case we have $c_v\ge 0$ for all $v\in\mathcal V$ and find an upper bound of $|\mathcal R_{\vec c}|$ as in the last paragraph. Now classical arguments allow to prove the existence of minimizers, and we skip them.
\end{proof}

We now state a characterization of the optimizers of \eqref{subdopt1}:

\begin{theorem}\label{optsd}
 Let $\mathcal V=\{v_1,\ldots,v_N\}\subset\R^d$ be such that \eqref{sumzero} holds, so that problem \eqref{subdopt1} has a solution. Then polyhedron $\mathcal R_{\mathcal V}$ realizing the maximum in \eqref{subdopt1} is unique up to translation, with face normals $v/|v|$ with $v\in\mathcal V$. Furthermore, there exists $\alpha_{\mathcal V}>0$ depending only on $\mathcal V$ such that for each $v\in\mathcal V$
 \begin{equation}\label{optsdeq}
 \text{The face $F_v$ of }\mathcal R_{\mathcal V}\text{ that has exterior normal }\frac{v}{|v|}\text{ has }\mathcal H^{d-1}(F_v)=\alpha_{\mathcal V} |v|.
 \end{equation}
 Furthermore, $C_{\mathcal V}=|\mathcal R_{\mathcal V}|=\frac{\alpha_{\mathcal V}}{d}$.
\end{theorem}
The main ingredients of the proof are the following:
\begin{itemize}
 \item The optimizer exists and is bounded due to Lemma \eqref{existcond}, and we call it $\mathcal R_{\vec c}$.
 \item For each $k=1,\ldots, N$ the hyperplane $\partial H_{v_k}(c_k)$ intersects $\partial R_{\vec c}$ because otherwise there would be the freedom to diminish $c_k$ and increase the $c_j, j\neq k$, which would increase the volume of $\mathcal R_{\vec c}$.
 \item We have $\frac{\partial R_{\vec c}}{\partial c_k}=\frac{\mathcal H^{d-1}(\partial H_{v_k}(c_k)\cap \mathcal R_{\vec c})}{|v_k|}$ for each $k$. Indeed if we increase $c_k$ by $\epsilon>0$ small enough, then $\mathcal R_{\vec c}$ will increase by a region which can be approximated as a prism with height $\epsilon/|v_k|$ and base $\partial H_{v_k}(c_k)\cap \mathcal R_{\vec c}$. For more details see \cite[7.2.2]{alexandrov2005convex}.
 \item Now the theorem follows by using Lagrange multipliers.
 \item Up to translation, $\mathcal R_{\mathcal V}$ contains the origin and then $c_k\ge 0$ for all $1\le k\le N$. Then $\mathcal R_{\mathcal V}$ is the union of pyramids with one vertex at the origin, heights $c_k/|v_k|$ and bases $F_k$. Thus we have 
 \[
 |\mathcal R_{\mathcal V}|=\frac1d\sum_{k=1}^N\frac{c_k}{|v_k|} \mathcal H^{d-1}(F_k)=\frac{\alpha_{\mathcal V}}d\sum_k c_k=\frac{\alpha_{\mathcal V}}d,
 \]
 which gives the last statement of the theorem.
\end{itemize}
For more details we again refer the reader to \cite[7.2]{alexandrov2005convex}. 
\begin{remark}[Computation of the maximum value $\alpha_{\mathcal V}/d$]
 We are not aware of a general formula for $\alpha_{\mathcal V}$ in terms of $\mathcal V$. This problem is equivalent to the one of computing the volume of a general convex polyhedron given its normals and face areas. See however iterative algorithms present in \cite{lasserre1983analytical} and for the $3$ dimensional case \cite{sellaroli2017algorithm}.

\end{remark}

We also note that a corollary of Theorem \ref{optsd}, we have Minkowski's famous theorem:
 \begin{theorem}[Minkowski]\label{mink}
   If $\nu_1,\ldots,\nu_N$ are unit vectors in an Euclidean space and $f_1,\ldots,f_N>0$ are positive numbers, then a convex polyhedron with face normals $\nu_1,\ldots,\nu_N$ and $d-1$-dimensional area of the face with normal $\nu_k$ equal to $f_k$ for $1\le k\le N$ exists if and only if 
   \begin{equation}\label{existsopt}
    \sum_{k=1}^Nf_i\nu_i=0,
   \end{equation}
and is unique up to translations.
 \end{theorem}
The connection of Theorem \ref{mink} to the setup of the rest of this section is that $\nu_k=v_k/|v_k|$ and \eqref{existsopt} then corresponds to \eqref{sumzero}. As mentioned before, Minkowski's proof idea was equivalent to the one approach presented here, and in a way our presentation just reformulates Minkowski's proof in our setting.

\subsection{Comparison to a related ``Wulff shape'' problem}

The condition $\sum_{v\in \mathcal V} c_v=1$ relates directly to our setup for involving the discrete graph laplacian. It is interesting to compare the minimization \eqref{subdopt1} itself to an isotropic isoperimetric problem. What seems the smallest perturbation of our problem in that direction would be to replace the constraint in \eqref{subdopt1} by the fixed-area constraint 
 \begin{equation}\label{anisoper}
\mathcal H^{d-1}\left(\partial\bigcap_{v\in\mathcal V}H_v(c_v)\right)=1.
 \end{equation} 
Then, if we optimize with respec to the $c_v$ the same quantity as in \eqref{subdopt1} with constraint \eqref{anisoper}, we find that the optimizer exists if and only if the one in the initial problem \eqref{subdopt1} exists, and minimizers of the new and old problem are related by translation and dilation by a factor depending only on $\mathcal V$. Then we could proceed as follows:
\begin{enumerate}
\item[1.] The problem constructed in the previous step has a solution then this is up to dilation also the solution of the dual problem 
\begin{equation}\label{isoper}
 \min\left\{\mathcal H^{d-1}(\partial R):\ \exists \vec c\in\R^{\mathcal V}, \ \mathcal R=\bigcap_{v\in\mathcal V}H_v(c_v), |\mathcal R|=1\right\}.
\end{equation}
\item[2.] Problem \eqref{isoper} can be rephrased as an anisotropic perimeter problem (below $\nu$ is a $\mathcal H^{d-1}$-measurable function which equals almost everywhere on $\partial R$:
\begin{subequations}\label{isoper2}\begin{equation}
 \min\left\{\int_{\partial R} F(\nu(x))\ d\mathcal H^{d-1}(x):\ \partial\mathcal R\mbox{ rectifiable, }  |\mathcal R|=1\right\},\end{equation}
 where
 \begin{equation}
 F:\mathbb S^{d-1}\to [0,+\infty],\qquad F(\nu)=\left\{\begin{array}{ll} 1&\mbox{ if }\nu\in\mathcal V,\\ +\infty&\mbox{ else.}\end{array}\right.
\end{equation}
\end{subequations}

\end{enumerate}
The solution to \eqref{isoper2} is given in e.g. in \cite{taylor} and the optumal $\mathcal R$ is the ``Wulff shape'' associated to the weight $F$, which is termed an ``extended integrand'' in \cite{taylor}. Explicitly, all solutions equal up to translation the set
\begin{equation}\label{solwulff}
  W_{\mathcal V}:=\bigcap_{v\in\mathcal V}H_{\frac{v}{|v|}}(1)=\bigcap_{v\in\mathcal V}H_{v}(|v|).
\end{equation}
The simple expression \eqref{solwulff} shows that in general faces of the isoperimetric optimizer will not have area $|v_k|$ as in Theorem \ref{optsd}, thus the Wulff and Minkowski optimization problems are quite different from each other.

\subsection{Application to the subdifferential optimization problem}

Coming back to the solution $u:\Omega\to\mathbb R$ of $\Delta_Au(x)\le C(\Omega,g)$ and $u(y)-u(x)=\frac{g(x,y)}{A(y,x)}$ for $(x,y)\in\overrightarrow{\partial\Omega}$, for fixed $x\in\Omega$ we apply Theorem \ref{optsd} with the following choices:
\begin{equation}\label{vx}
 \mathcal V_x:= \{(y-x)A^2(x,y):\ A(x,y)\neq 0\},
\end{equation} 
and for $v=(y-x)A^2(y,x)\in\mathcal V_x$ define 
\begin{equation*}
c_v=\frac{A^2(y,x)(u(y)-u(x))}{\Delta_Au(x)}.
\end{equation*}
Then we have (note that $\lambda\cdot A=\{\lambda x: \ x\in A\}$ if $A\subset \mathbb R^d, \lambda\in\mathbb R$ is the dilation of a set)
\[
\bigcap_{v\in\mathcal V_x}H_v(c_v)=\frac{1}{\Delta_Au(x)}\cdot\partial^{\mathrm{prox}}u(x),\quad\quad \sum_{v\in\mathcal V_x}c_v=1,
\]
and with these conversions, from the results of the previous section we obtain with little effort the below:
\begin{proposition}\label{equalitysubdiff}
With notations \eqref{subdopt1} and \eqref{vx}, for any function $u:\{x\}\cup\{y:\ x\sim y\}\to \mathbb R$ there holds
 \begin{equation}\label{eqsd}
 |\partial^{\mathrm{prox}}u(x)| \le C_{\mathcal V_x}(\Delta_Au(x))^d.
 \end{equation}
The following are equivalent for a function $u$ as above, if $F_{x,y}^{\mathrm{prox}}$ denotes the face of $\partial^{\mathrm{prox}}u(x)$ with normal vector $\frac{y-x}{|y-x|}$:
\begin{enumerate}
\item[(a)] The function $u$ as achieves the equality in \eqref{eqsd}. 
\item[(b)] For each $y\sim x$ there holds 
\begin{equation}\label{condeqsd}
\frac{\mathcal H^{d-1}(F_{x,y}^{\mathrm{prox}})}{|\Delta_A u(x)|^{d-1}}= d\ C_{\mathcal V_x}|y-x|A^2(x,y).
\end{equation}
\item[(c)] There exists a constant $\tilde \alpha_{\mathcal V_x}$ such that for each $y\sim x$ we have
\begin{equation}\label{alphavx}
 \frac{\mathcal H^{d-1}(F_{x,y}^{\mathrm{prox}})}{|y-x|A^2(x,y)}=\tilde\alpha_{\mathcal V_x}.
\end{equation}
\end{enumerate}

Moreover the following conditions are equivalent:
 \begin{enumerate}
     \item There exists $u$ as above such that equality is achieved in \eqref{eqsd}.
     \item There holds $\mathrm{Span}\{y-x:\ y\sim x\}=\mathbb R^d$ and $\sum_{y:y\sim x}(y-x)A^2(x,y)=0$.
     \item There holds $\mathrm{Span}\{y-x:\ y\sim x\}=\mathbb R^d$ and $\Delta_A$ has linear precision at $x$, i.e. $\Delta_A\ell(x)=0$ for each $\ell:\{x\}\cup\{y:\ y\sim x\}\to \mathbb R$ which is the restriction of an affine function.
 \end{enumerate}
 Moreover, if $u$ achieving the equality in \eqref{eqsd} exists, then it is unique up to summing to it the restriction of an affine function.
\end{proposition}
\begin{proof}
Translating the condition \eqref{optsdeq} for the optimal shape 
\[
\mathcal R_{\mathcal V_x}=\frac{1}{|\Delta_Au(x)|}\cdot\partial^{\mathrm{prox}}u(x),
\]
 we obtain the equivalence $(a)\Rightarrow (b)$.  The uniqueness of the optimal shape, from Theorem \ref{optsd}, gives the opposite implication. The equivalence between \eqref{condeqsd} and \eqref{alphavx} is consequence of the condition \eqref{optsdeq} from Theorem  \ref{optsd}, in which again we take $\mathcal V=\mathcal V_x$ as defined in \eqref{vx}, and $\tilde\alpha_{\mathcal V_x}=\alpha_{\mathcal V_x}/|\Delta_Au(x)|^{d-1}$. 
 
\medskip

The equivalence of (1), (3) follows Lemma \ref{existcond}, and the equivalence of (2), (3) follows from Theorem \ref{optsd} as explained before the statement of the proposition. 

For the equivalence of (3) and (4) we note that $\Delta_A$ vanishes on constant functions for all $A$, thus we may equivalently verify the vanishing condition in (4) only for linear functions. Now condition (3) states that $\Delta_Ax=0$ where $\Delta_A$ is applied componentwise to the identity, whose components form a basis for the space of linear functions $\ell:\mathbb R^d\to \mathbb R$, which implies its equivalence to (4).

The uniqueness statement about $u$ follows from the uniqueness part of Theorem \ref{optsd}, noting that $\Delta_A(u+\ell)=\Delta_Au$ if $A$ satisfies (4), and that $\partial^{\mathrm{prox}}(u+\ell)(x)=\partial^{\mathrm{prox}}u(x) + p$ if $\ell(x)=p\cdot x+ x_0$ is an affine function.
\end{proof}

\section{Discrete isoperimetric equality conditions}\label{sec5}
As a consequence of the study of \eqref{proofscheme} we find that equalities at all inequalities can only hold if $H_g$ is partitioned into the $\partial^{\mathrm{prox}}u(x),x\in\Omega$. In turn, by Proposition \ref{equalitysubdiff}, the $\partial^{\mathrm{prox}}u(x)$ need to all have faces $F_{x,y}$ perpendicular to the directions $y-x$ for $x\sim y$, and this is only possible if $\Delta_A$ has linear precision. If the last inequality becomes equality in \eqref{proofscheme} then we have $\Delta_Au$ constant over $\Omega$, and thus all $\partial^{\mathrm{prox}}u(x)$ have equal volumes too. 

\medskip

We note that the condition that $\Delta_AF=0$ for affine functions $F$ can also be rephrased as saying that $A^2(x,y)|x-y|$ is a tension over the graph with vertices $V$ and edges given by $\{x,y\}$ such that $A(x,y)\neq 0$. This condition was first studied in detail by Cremona \cite{cremona1879figure} and Maxwell \cite{maxwell1864xlv}, and we refer to \cite{rybnikov} and the references therein, for a generalized setup. We here collect results from the theory of stresses and tensions, which help to settle the existence and theorically allow to determine isoperimetric shapes either with weights or without weights.

\medskip

We introduce in Section \ref{survey} a series of definitions and a survey of classification results that allow to connect distinct interpretations of our semidiscrete data. The results cited in this section are not used in our proofs but seem worth to mention for the sake of creating a more complete picture of the data we are discussing in the proofs. In Section \ref{sec53} we use the definitions from Section \ref{sec53} as well as results from all the previous sections, and prove our main result, Theorem \ref{specialiso}.

\subsection{Survey: Isomorphisms between Reciprocals, Liftings, Pogorelov and Aleksandrov solutions, for general rectilinear cell decompositions}\label{survey}\hfill

We introduce here some results of Rybnikov \cite{rybnikov} and Aurenhammer \cite{aurenhammer1987criterion} related to our setting, as a way of introducing semidiscrete optimal transport notation. The connection of this setup to Optimal Transport appears in \cite{merigot},\cite{benamoufroese}.

\medskip

We invite the reader to consult \cite{rybnikov}, who develops the very elegant fundamental setup in which one can define ad-hoc obstructions for the extension of given objects (such as convex functions, dual complexes, suitable liftings, etc.) around combinatorial cycles, and in providing conditions under which they vanish. Note that Rybnikov's results are available for general piecewise-linear realizations of manifolds with vanishing first $\mathbb Z/2\mathbb Z$-homology group and satisfying local combinatorial non-degeneracy conditions. For simplicity (especially, in order to avoid all discussions about fixing orientations) we restrict to the case of $\mathbb R^d$ here, although extensions of our methods and results to Rybnikov's more general setting would be also possible. 

\medskip

Here are the notations and definitions that we need. In the below, $\mathcal K$ will denote a $d$-dimensional simplicial complex and $K$ will be its piecewise-linear realization of $\mathcal K$ in $\mathbb R^d$. By this we mean that each $k$-dimensional cell of $K$ is a $k$-dimensional polyhedron identified bijectively to a $k$-cell in $\mathcal K$, so that the combinatorial cell complex structure of $\mathcal K$ is conjugated by this bijection to the piecewise affine cell complex structure of $K$. We assume that orientations are defined on cells of $K$ and are compatible with the combinatorial boundary operators. Note that no immersion or embedding requirements are made on this identification at this point. A further regularity condition is needed on our complexes. We only apply the below results for the case of convex bounded $k$-cells, however Rybnikov's weaker assumption is that each $k$-cell has a homology $(k-1)$-sphere as boundary.
 
\begin{itemize}

\item A {\bf lifting of $K$} is an assignment to each $k$-cell $C$ of $K$, for all $k=0,\ldots,d$, of an affine $\mathbb R$-valued function $L(C)$, so that $L(C)|_F=L(F)$ for each $F\in\partial C$. A lifting is called {\bf locally convex} if for each subcomplex $K'$ of $K$ which is embedded into $\mathbb R^d$, the restrictions of $L$ to cells in $K'$ coincide with a convex function. A lifting is called {\bf sharp} if $\nabla L(C)\neq \nabla L(C')$ whenever $C,C'$ are adjacent $d$-cells in $K$.


\item The {\bf combinatorial dual graph} of $\mathcal K$, denoted $\mathcal{G}(\mathcal K)$, is a graph whose vertices correspond to $d$-cells of $\mathcal K$, and whose edges correspond to internal $(d-1)$-cells of $\mathcal K$.

\item A {\bf reciprocal of $K$} is a piecewise-linear realization $R$ in $\mathbb{R}^d$ of the combinatorial dual graph  $\mathcal{G}(\mathcal K)$ such that the edges of $R$ are segments perpendicular to the corresponding $(d-1)$-dimensional facets of $K$. If none of the edges of a reciprocal collapses into a point, then the reciprocal is called {\bf non-degenerate}. We assume that an orientation on $R$ is defined and is compatible with the orientation on $K$, and if furthermore $R$ is an embedding then we call it {\bf convex}.
\end{itemize}

\begin{theorem}[convex liftings $\simeq$ convex reciprocals {\cite[Thm. 6.1]{rybnikov}}]\label{reclift}
If $K$ is a face-to-face polyhedral decomposition of $\mathbb R^d$ and $L: K\to \mathrm{Aff}(d)$ is a lifting of $K$, then to $L$ we associate those realizations of $\mathcal G(\mathcal K)$ that satisfy the property 
\[
 R(C)-R(C')= \nabla L(C)-\nabla L(C')\quad\text{whenever}\quad C,C'\quad\text{are adjacent $d$-cells of $K$.}
\]
This association establishes a bijection between reciprocals of $K$ defined up to a translation and liftings of $K$ defined up to summing an affine function. Moreover, sharp liftings correspond to non-degenerate reciprocals, and convex liftings correspond to convex reciprocals under the above association.
\end{theorem}

\begin{itemize}
\item Given a discrete point configuration $V\subset \mathbb{R}^d$ and a map $w:V\to\mathbb R$, the associated \textbf{weighted Voronoi diagram} (also found under the names of: {\bf power diagram}, {\bf Laguerre diagram}, {\bf weighted Dirichlet decomposition}) is the collection of cells, indexed by $V$, where the cell associated to a given $p\in V$ is given by
\begin{equation}\label{vorocell}
C(p):=\{x\in\mathbb R^d:\ |p-x|^2-w(p) \leq |q-x|^2-w(q)\text{ for all $q \in V$}\}.
\end{equation}
Note that $C(p)$ is necessarily convex for all $p\in V$.
\end{itemize}

The proof idea for the second part of the following is present in Aurenhammer \cite{aurenhammer1987criterion}, and Rybnikov \cite[Thm 7.2 and the following observation]{rybnikov} notes that it can be extended to the present setting:

\begin{theorem}[convex liftings and weighted Voronoi diagrams {\cite{aurenhammer1987criterion},\cite{rybnikov}}]\label{liftvor}
Let $K$ be a face-to-face decomposition of $\mathbb R^d$ into convex finite polyhedra, and that $L:K\to \mathrm{Aff}(d)$ is a lifting. Then the following hold:
\begin{enumerate}
\item $L$ is locally convex if and only if there exists a convex function $\phi:\mathbb R^d\to\mathbb R$ such that $\phi|_C=L(C)|_C$ for each cell $C$ of $K$.
\item If $R$ is a reciprocal associated to $L$ as in Theorem \ref{reclift} and $V=R(\{d-\text{cells of }K\})$ are the vertices of $R$, then we can define associated weights $w:V\to[0,+\infty)$ so that the weighted Voronoi cells $\{C(p): p\in V\}$ are the $d$-cells of $K$. This association iduces a bijection between convex liftings defined up to an affine summand and weights defined up to additive constants.
\end{enumerate}
\end{theorem}

Finally, we note the connection to classical Optimal Transport theory. In optimal transport theory the result is usually proved for the case of finite $V$ (or $K$), see e.g. \cite[Thm. 2.3]{benamoufroese}, or \cite[Thm. 3 ]{merigot}, as well as the presentation in Section \ref{aleksintro}. The original idea from \cite[Thm. 2.1]{aurenhammer1987criterion} directly generalizes to the infinite case, which allows a more general connection to the setting we are treating in our survey. 

\begin{theorem}[convex lifitings and semidiscrete Optimal Transport solutions]\label{vorot}
Assume that $K$ forms a face-to-face decomposition of $\mathbb R^d$ into convex finite polyhedra. Let $\phi:\mathbb R^d\to\mathbb R$ be a piecewise affine convex lifting of $K$ and $R$ be a convex reciprocal of $K$ corresponding to $\phi$ according to Theorem \ref{reclift}. Then the following hold:
\begin{enumerate}\item There exists a unique convex function $u:\mathbb R^d\to \mathbb R$ such that the subdifferentials satisfy $\partial u(p)=C(p)$ and that $\phi(x)=\langle x, p\rangle - u(p)$ for all $p\in V$ and $x\in C(p)$. Furthermore $u$ is its Legendre-Fenchel transform of $\phi$. 

\item If $V$ is the vertex set of $R$ then for each $V'\subset V$ finite, $\partial u$ is the unique optimal transport map from the measure $\sum_{p\in V'}|C(p)|\delta_p$ to the restriction of the Lebesgue measure to $\bigcup_{p\in V'}C(p)$, with respect to the cost $|x-y|^2$.
\end{enumerate}
\end{theorem}
We mention the following Optimal Transpont terminology, to facilitate the comparison to \cite{benamoufroese}: 
\begin{itemize}\item $\phi$ as in Theorem \ref{vorot} is known as {\bf Pogorelov solution} to the optimal transport problem mentioned in point 2 of the theorem.
\item $u$ as in Theorem \ref{vorot} is known as the {\bf Aleksandrov solution} to the same optimal transport problem.
\end{itemize}

\subsection{Discrete isoperimetric inequality and discussion of equality cases}\label{sec53}
We now formulate sufficient conditions for the existence of the optimal isoperimetric shape. These conditions will be further discussed and accompanied by several examples in Section \ref{s6}. Note that only the case $g=1$ will be explored in details, however the result for general $g$ may be of interest and requires no essential extra effort.
\begin{theorem}\label{specialiso}
Let $V\subset\mathbb R^d$ be a discrete set and $A:V\times V\to [0, +\infty)$ be a symmetric nonnegative weight. Let $G$ be the graph with vertices $V$ and edges $E_A:=\{(x,y):\ A(x,y)\neq 0\}$. Assume that condition \eqref{nbdcond} holds at each $x\in V$ and that $G$ is locally finite. Let $\Omega\subset V$ be a finite subset and let $\overline\Omega$ be the closure of $\Omega$ in $G$. Let $g:E_A\to[0,+\infty)$ be a second weight, and set $\sharp_g\overrightarrow{\partial \Omega}:=\sum_{(x,y)\in\overrightarrow{\partial\Omega}}g(x,y)$, where $\sharp_g\overrightarrow{\partial \Omega}=\sharp \overrightarrow{\partial\Omega}$ for $g\equiv 1$.

\medskip

\noindent Let $u:\overline{\Omega}\to \mathbb R$ be a solution, in the sense \eqref{newbdcond}, of
\begin{equation}\label{ubdp}
 \left\{\begin{array}{ll}
  \Delta_Au(x)\le \frac{\sharp_g\overrightarrow{\partial\Omega}}{\sharp \Omega}&\text{ if }x\in\Omega,\\[3mm]
  \partial_\nu u(x)=\frac{g(x,y)}{A(x,y)}&\text{ if }(x,y)\in\overrightarrow{\partial\Omega}.
 \end{array}\right.
\end{equation}
Then we have, with notations \eqref{subdopt1} and \eqref{vx}, 
\begin{equation}\label{ineqchain1}
 \left|\partial u(\Omega)\right|\le \sum_{x\in\Omega}|\partial^{\mathrm{prox}}u(x)| \le \sum_{x\in\Omega}C_{\mathcal V_x}\left(\Delta_Au(x)\right)^d \le \left(\max_{x\in\Omega} C_{\mathcal V_x}\right) \frac{(\sharp_g\overrightarrow{\partial\Omega})^d}{(\sharp\Omega)^{d-1}}.
\end{equation}
{\bf Necessary conditions for equality.} If equality holds in \eqref{ineqchain1} then all $C_{\mathcal V_x}, x\in\Omega$ are equal, $G|_{\overline\Omega}$ is the reciprocal of a face-to-face decomposition of $\partial u(\Omega)$ into convex polyhedra $\partial^{\mathrm{prox}}u(x),x\in\Omega$ of equal volume, and $u$ satisfies $\Delta_Au(x)=\sharp_g\overrightarrow{\partial\Omega}/\sharp\Omega$ in $\Omega$, $u(y)-u(x)=\frac{g(x,y)}{A(x,y)}$ for all $(x,y)\in\overrightarrow{\partial\Omega}$.
\medskip

\noindent {\bf Sufficient conditions for equality.} Let $\Omega\subset V$ be connected within $G$. Equality in the isoperimetric inequality in graph $G$ with weight $A$ is achieved by $\Omega$ if the following geometric conditions are met by $G,A$:
\begin{enumerate}
 \item The complex made of vertices and edges of $G|_{\Omega}\cup\overrightarrow{\partial\Omega}$ (note that vertices in $\overline{\Omega}\setminus\Omega$ are not included) is reciprocal to the collection of $d$- and $(d-1)$-cells of an equal-volume face-to-face polyhedral tessellation $K$ of a convex polyhedron $H$.
 \item If $F_{x,y}$ denotes the $(d-1)$-dimensional facet of the tessellation $K$ which is dual to edge $(x,y)$ of $G$, then the $\frac{A^2(x,y)|y-x|}{\mathcal H^{d-1}(F_{x,y})}$ takes the same value for all edges of $G$.
\end{enumerate}
 Under the above conditions, functions $u$ achieving equality in \eqref{ineqchain1} are precisely those of the form $u=\lambda u_{\mathrm{Alek}} + \ell$ where $\lambda>0$, $\ell$ is an affine function and $u_{\mathrm{Alek}}$ is the Aleksandrov solution to the optimal transport problem between $\frac{|H|}{\sharp\Omega}\sum_{x\in\Omega}\delta_x$ and $1_H(x)dx$ with cost $|x-y|^2$.
\end{theorem}
\begin{proof}\hfill

{\bf 1.} {\it Proof of \eqref{ineqchain1}.}
The inclusion $\partial u(x)\subset \partial^{\mathrm{prox}}u(x)$ follows because the latter is defined by a subset of the constraints of the former. This in turn implies the first inequality in \eqref{ineqchain1}. Proposition \ref{equalitysubdiff} gives the second inequality in \eqref{ineqchain1}. The last inequality in \eqref{ineqchain1} follows by estimating $C_{\mathcal V_x}$ by its maximum over $x\in\Omega$ and $\Delta_Au(x)$ by its upper bound imposed in \eqref{ubdp}.

\medskip

{\bf 2.} {\it Necessary conditions for equality in \eqref{ineqchain1}.}
If equality holds in the first inequality in \eqref{ineqchain1} then we must have that $\partial u(\Omega)$ si the disjoint union of the $\partial u(x),x\in\Omega$ and that for all $x\in\Omega$ there holds $\partial u(x)=\partial^{\mathrm{prox}}u(x)$. Therefore $\partial u(\Omega)$ is decomposed into the convex polyhedra $\partial^{\mathrm{prox}}u(x)$. Moreover, due to Proposition \ref{proxsubdiff}, we have that $u$ is the restriction of a convex function defined on $\mathbb R^d$ and it follows that $\Delta_Au(x)\ge 0$ for all $x\in \Omega$.

\medskip

If the last inequality in \eqref{ineqchain1} is an equality then $\Delta_Au(x)$ is constant over $x\in\Omega$ and $C_{\mathcal V_x}$ is independent of $x\in\Omega$. Moreover, due to the divergence theorem, equality in \eqref{ubdp} is achieved only if $u(y)-u(x)=\frac{g(x,y)}{A(x,y)}$ for all $(x,y)\in\overrightarrow{\partial\Omega}$. If the middle inequality in \eqref{ineqchain1} is an equality then for all $x\in \Omega$ we have $|\partial^{\mathrm{prox}}u(x)|=C_{\mathcal V_x} (\Delta_Au(x))^d$, and by the above, is independent of the choice of $x\in\Omega$.
\medskip

{\bf 3.} {\it Sufficient conditions for equality in \eqref{ineqchain1}.} Assume now that $G,A,H$ are as in the last part of the theorem, and denote by $\{C(x):\ x\in V\}$ the equal-volume tessellation of point (1) of the theorem statement. We start with two observations, which are useful later.

\medskip

Firstly, note that by Theorem \ref{optsd}, if $\mathcal H^{d-1}(F_{x,y})/A^2(x,y)|y-x|=\alpha$ are all equal, then the optimal constants $C_{\mathcal V_x}$ satisfy $C_{\mathcal V_x}=\frac{\alpha}d$ and thus are all equal for $x\in \Omega$ as well.

Secondly, let $u:\overline{\Omega}\to\mathbb R$ be a solution to $\Delta_A u(x) = \frac{\sharp_g\overrightarrow{\partial\Omega}}{\sharp \Omega}$ for $x\in\Omega$ and $u(y)-u(x)=\frac{g(x,y)}{A(x,y)}$ for $(x,y)\in\overrightarrow{\partial\Omega}$. Such $u$ exists because the complex $G|_{\Omega}\cup\overrightarrow{\partial\Omega}$ is reciprocal to the facets of dimension $d$ and $d-1$ of the tessellation $\{C(x):\ x\in \Omega\}$ of $H$, and thus the edges $(x,y)\in \overrightarrow{\partial\Omega}$ are orthogonal to faces of $H$, which is convex. In this case we cannot have an $y\in\overline{\Omega}\setminus\Omega$ participating to two different edges $(x,y),(x',y)\in\overrightarrow{\partial\Omega}$, and we can apply Lemmas \ref{naiveneumbd} and \ref{naivegood}. This uniquely determines $u$ up to summands $\ell$ such that $\Delta_A \ell=0$.
\medskip

We now show that a suitable rescaling of the Aleksandrov solution $u_{\mathrm{Alek}}$ described in the theorem is a solution as described in the above paragraph, for some function $g$ determined by $G,A$. Indeed, if $\partial^{\mathrm{prox}}u_{\mathrm{Alek}}(x)=C(x)$ for all $x\in\Omega$ then $u_{\mathrm{Alek}}$ satisfies \eqref{alphavx} of Proposition \ref{equalitysubdiff} at $x$, and therefore by the equivalence of (a), (b), (c) in that proposition, it satisfies the equality in the middle inequality in \eqref{ineqchain1} and also satisfies \eqref{condeqsd} from the proposition. Since $u_{\mathrm{Alek}}$ is convex we have $\Delta_Au_{\mathrm{Alek}}\ge 0$ and then \eqref{condeqsd} and the fact that $C_{\mathcal V_x}$ are constant imply that $\Delta_Au_{\mathrm{Alek}}$ is also constant on $\Omega$.

\medskip

We now show that any function $u:\oo\to\mathbb R$ that realizes the equality in \eqref{ineqchain1} must be $\lambda u_{\mathrm{Alek}}+\ell$ for some $\lambda>0$ and $\ell:\mathbb R^d\to\mathbb R$ affine. 

By the already proved necessary conditions for equality, we have $\Delta_A u(x)$ constant. Since furthermore $C_{\mathcal V_x}$ and $\frac{A^2(x,y)|y-x|}{\mathcal H^{d-1}(F_{x,y})}$ are constant for $x$ and $(x,y)$ in $G_{\overline{\Omega}}$, we have that after rescaling by a constant factor, $u$ satisfies the condition \eqref{condeqsd} for equality of Proposition \ref{equalitysubdiff} at every $x\in \Omega$, and thus $|\partial^{\mathrm{prox}}\lambda u(x)|=C_{\mathcal V_x}(\Delta_A \lambda u(x))^d$. By dividing by $\lambda^d$, the same equality also holds for $u$ itself. As a consequence, the middle equality in \eqref{ineqchain1} also holds for $u$. Proposition \ref{equalitysubdiff} also implies that $\Delta_A\ell=0$ for any affine $\ell$.

\medskip

\noindent Furthermore, by Proposition \ref{equalitysubdiff}, $u$ is now fixed up to affine summands and facets and facet directions of $\partial^{\mathrm{prox}} \lambda u(x)$ coincide with those of $C(x)$. Thus, $\partial^{\mathrm{prox}} \lambda u(x)$ is a translation of $C(x)$, and this holds for all $x\in\Omega$. Furthermore, a facet which is a translation of $F_{x,y}$ is shared between $\partial^{\mathrm{prox}}\lambda u(x)$ and $\partial^{\mathrm{prox}}\lambda u(y)$ for all $(x,y)$ edge of $G$. This implies that there exists a $t\in\mathbb R^d$ such that $\partial^{\mathrm{prox}}\lambda u(x)=C(x)+t$ for all $x\in\Omega$. It follows that after removing a suitable affine summand from $\lambda u$, we have 
\begin{equation}\label{proxsubduua}\partial^{\mathrm{prox}}\lambda u(x) = \partial u_{\mathrm{Alek}}(x)=\partial^{\mathrm{prox}}u_{\mathrm{Alek}}(x)\quad\text{at all }\quad x\in\Omega,
\end{equation}
where $u_{\mathrm{Alek}}$ is the Aleksandrov solution. The above second equality follows because $C(x)$ has facets $F_{x,y}$ in bijection with edges $(x,y)$ of $G$, and as a consequence only vectors $y-x$ with $(x,y)$ edge of $G$ need to be included as constraints while defining the subdifferential $\partial u_{\mathrm{Alek}}(x)$. 

\medskip

\noindent We claim that due to \eqref{proxsubduua} and to the fact that facets $F_{x,y}$ are in bijection with edges $(x,y)$ in $G$ and to the assumption that $\Omega$ is connected, it follows that $u, u_{\mathrm{Alek}}$ are equal up to a constant summand. Indeed, start at a vertex $x_0$ and add a constant to $u$ so that $u(x_0)=u_{\mathrm{Alek}}(x_0)$. Then following the edges from $x_0$ to another vertex $x\in\Omega$, we find that $u$ and $u_{\mathrm{Alek}}$ must have equal jumps along each edge due to \eqref{proxsubduua} and the claim follows.

\end{proof}

Note that the left-hand side of \eqref{ineqchain1} depends on $\Omega$ through the discrete Neumann boundary value problem \eqref{ubdp}, a somewhat indirect dependence. As mentioned in the introduction, a way to obtain a more explicit description of the equality cases for a restricted class of examples, is to complete \eqref{ineqchain1} on the right by a further inequality, coming from the inclusion
\begin{equation}\label{ineqsubdiff}
H_g:= \left\{p\in\mathbb R^d:\ \forall (x,y)\in\overrightarrow{\partial\Omega},\ p\cdot (y-x)\le\frac{g(x,y)}{A(x,y)}\right\}\quad \subseteq\quad \partial u(\Omega).
\end{equation}
The fact that this holds for solutions of \eqref{ubdp} due to the maximum principle is explained in the introduction, see the proof of inequality (a) from \eqref{proofscheme}. Assuming that $\partial u(\Omega)$ is convex, a necessary and sufficient condition for equality in \eqref{ineqsubdiff} is that 
\begin{equation}\label{neccondineq2} 
\left.\begin{array}{cc}(x,y), (x',y')\in\overrightarrow{\partial\Omega}\\[3mm]\frac{y-x}{|y-x|}=\frac{y'-x'}{|y'-x'|}\end{array}\right\}\quad \Rightarrow\quad\frac{A(x,y)|y-x|}{g(x,y)}=\frac{A(x',y')|y'-x'|}{g(x',y')}.
\end{equation}
 The reasoning concerning \eqref{proofscheme} together with the above and the sufficiency condition in Theorem \ref{specialiso} now give the following.
\begin{proposition}\label{corstrong}
 Let $V,G,A,g$ be as in in the beginning of Theorem \ref{specialiso}. Assume that the following further properties hold:
 \begin{enumerate}
  \item $G$ reciprocal to an equal-volume triangulation of $\mathbb R^d$.
  \item There exist constants $C_1, C_2>0$ such that 
  \begin{equation}\label{conststuf}
  \frac{\mathcal H^{d-1}(F_{x,y})}{|x-y|A^2(x,y)}=C_1\quad \text{and} \quad \frac{\mathcal H^{d-1}(F_{x,y})|x-y|}{g^2(x,y)}=C_2
  \end{equation}
 for all edges $(x,y)$ of $G$ and all corresponding dual $(d-1)$-facets of the triangulation.
 \end{enumerate}
Then there exists $C_{G,A}>0$ such that with notations \eqref{subdopt1} and \eqref{vx}, 
 \[
 C_{\mathcal V_x}=C_{G,A}\quad\text{for all}\quad x\in V.
 \]
Furthermore, for each finite subset $X\subset V$, denoting $H_{X,A}:=\left\{p\in\mathbb R^d:\ \forall (x,y)\in\overrightarrow{\partial X},\ A(x,y) (y-x) \cdot p\le g(x,y)\right\}$, there holds
\begin{equation}\label{isopineq}
(\sharp X)^{d-1}\ \le\ \frac{C_{G,A}}{ \left|H_{X,A}\right|} (\sharp_g\overrightarrow{\partial X})^d,
\end{equation}
with equality if and only if the cells dual to points in $X$ tessellate a dilation of $H_{X,A}$.

\medskip

In particular, if we take
 \begin{equation}\label{hlattice}
 H_g:=\{p\in \mathbb R^d:\ \forall x,y\in V \text{ such that } A(x,y)\neq 0,\ A(x,y)(y-x) \cdot p\le g(x,y)\},
 \end{equation}
then for all $X\subset V$ we have 
\begin{equation}\label{isopineqall}
(\sharp X)^{d-1}\ \le\ \frac{C_{G,A}}{|H|}\ (\sharp_g\overrightarrow{\partial X})^d,
\end{equation}
with equality if and only if a multiple of $H$ is tessellated by simplices dual to points in $X$.
\end{proposition}
Equation \eqref{isopineq} transforms into an inequality which contemporarily holds for all sets $X$, only if we take the supremum on $X$ over the constants in the right hand side. This corresponds to verifying \eqref{isopineqall} for the case of $H$ which is the intersection of all possible $H_{X,A}$. Note that equality in \eqref{isopineqall} is in general hard to achieve, however we find several specific examples in Section \ref{s6}.

\begin{proof}[Proof of Proposition \ref{corstrong}:]
Theorem \ref{specialiso} proves the ``only if'' part of the implications claimed in the theorem, thus we focus on proving only the ``if'' part, i.e. the sufficiency of conditions \eqref{isopineq} and respectively \eqref{isopineqall}, for optimality. We restrict to the latter only, as the proof is similar for the two cases.

\medskip

{\bf Important observation.} For each $x\in V$ the optimizer to $C_{\mathcal V_x}$ is a simplex with facets orthogonal to vectors in $\mathcal V_x$, and crucially, all $\partial^{\mathrm{prox}}u(x)$ is then a dilated and translated copy of this optimizer. 
 
 \medskip
 
Assume that $\Omega$ has dual cells tessellating a multiple of $H$. Due to the above highlighted observation, we can proceed in the opposite way than for the proof of sufficiency in Theorem \ref{specialiso}, and consider a solution $u$ to the boundary value problem $\Delta_Au(x)=\frac{\sharp_g \overrightarrow{\partial\Omega}}{\sharp\Omega}$ with $u(y)-u(x)=\frac{A(x,y)}{g(x,y)}$ for all $(x,y)\in\overrightarrow{\partial\Omega}$. The convexity of $H_g$ implies that in fact this na\"ive version of the discrete Neumann boundary condition coincides with the Hamamuki one in this case, similarly to the sufficiency proof in Theorem \ref{specialiso}.

\medskip

From our highlighted observation, 
\begin{equation}\label{facetquotient}
\exists c>0\quad \text{ such that }\forall x\in V,\ \forall y\sim x\qquad \frac{\mathcal H^{d-1}(F_{x,y})}{\mathcal H^{d-1}(F^{\mathrm{prox}}_{x,y})}=c_x.
\end{equation}
Our observation gives the existence of $c_x$, possibly dependent on $x\in V$, for which \eqref{facetquotient} holds. However, since $G$ is connected, and since \eqref{facetquotient} would give $c_x=c_y$ if $x,y$ form an edge in $G$, it follows that such $c_x$ actually is independent of $x$, proving \eqref{facetquotient} as stated.

\medskip

Like in the proof of the sufficiency in Theorem \ref{specialiso}, we find that $C_{\mathcal V_x}$ are all equal at all $x\in V$. Now, using the fact that also $\frac{H^{d-1}(F_{x,y})}{A^2(x,y)|y-x|}$ and $\Delta u_A(x)$ are constant, it follows from \eqref{facetquotient} that condition \eqref{condeqsd} is satisfied by $u$ at each $x\in\Omega$. Thus the last two inequalities in \eqref{ineqchain1} becomes an equality for our choice of $u$.

\medskip

We then proceed as in the last part of Theorem \ref{specialiso} to prove that $u=\lambda u_{\mathrm{Alek}} + \ell$, where we may subtract $\ell$ since again $\Delta_A\ell=0$ for affine $\ell$. This implies that a $\frac{1}{\lambda} H_g$ is tessellated by $\partial^{\mathrm{prox}}u(x)$ for $x\in\Omega$, completing the proof.
\end{proof}

\section{Isoperimetric inequalities in periodic graphs}\label{s6}

In this section we will present a few examples, in which the hypotheses of Proposition \ref{corstrong} are satisfied and in which equality in \eqref{isopineqall} can be achieved and characterized. We do not aim at exhausting all possibilities, but rather at giving a hint at some possible applications. A more thorough study and classification is left to future work.

\medskip

For each case we will present $V,A,g$, the associated constant $C_{G,A} $ from Proposition \ref{corstrong} point (3), the region $H$ as in \eqref{hlattice} and its volume.

\medskip

In some cases, obstructions to the achievability of equality in \eqref{isopineqall} appear, while a discrete isoperimetric shape indeed exists nevertheless, leaving open possible directions of improvement of the theory. This suggests that alternative criteria alternative to Theorem \ref{specialiso} and Proposition \ref{corstrong} may exist. However such improvements go beyond the use of our PDE/Semidiscrete Optimal Transport approach, and thus we do not pursue them here.

We also restrict to periodic graphs only, and do not explore the whole range of possibilities left open by Theorem \ref{specialiso} and Proposition \ref{corstrong}. 

\medskip

We emphasize that even only restricting to the periodic case, the range of possible examples is very large, including the previous work \cite{hamamuki} as a very special case. The whole classification of the periodic graphs in which Proposition \ref{corstrong} applies to yield sharp isoperimetric inequality in \eqref{isopineqall} goes beyond the scope of this work, and leads to open questions of independent interest, of which we mention a few.

\medskip

Besides a range of specific examples, we present three general methods for producing new cases of applicability of Proposition \ref{corstrong} from known ones, by either taking products, enriching the graph by subdividing the dual polyhedral tessellation, or stacking periodic graphs along new coordinate directions. 
\subsection{Examples based on tilings of $\mathbb R^2$ by triangles or parallelograms}
\subsubsection{Tiling by equilateral triangles}\label{trix}

Consider the triangular lattice in $\mathbb{R}^2$ given by
\[
 \Lambda:=(1,0)\mathbb Z+\left(\tfrac12,\tfrac{\sqrt3}2\right)\mathbb Z:=\left\{x\in\mathbb R^2:\ \exists m,n\in \mathbb Z, \ x=m(1,0)+n\left(\tfrac12,\tfrac{\sqrt3}{2}\right)\right\}.
\]
We let $K$ be the tessellation of $\mathbb R^2$ formed by the equilateral triangles of sidelength $1$ with vertices in $\Lambda$. To define a rectilinear graph reciprocal to $K$, take vertex set $V$ equal to the set of barycenters of all tiles from $K$, and as edges $[x,y]$ the segments joining $x\neq y \in V$ which correspond to tiles in $K$ having a face in common. The same $G$ can be described as the $1$-skeleton of the Voronoi diagram generated by $\Lambda$, and we name it the {\bf honeycomb graph}. We now take $A,g$ symmetric and such that $A(x,y)=g(x,y)=1$ for all $[x,y]$ edge of $G$, $A(x,y)=g(x,y)=0$ for other pairs $(x,y)\in V\times V$. 

\medskip

The hypotheses of Proposition \ref{corstrong} can be easily checked because the faces $F_{x,y}$ of tiles of $K$ all have length $1$ ad the lengths of edges of $G$ are all equal to $\sqrt3/3$. In particular, note that for each $x\in V$, the set $\mathcal V_x=\{(y-x)A^2(x,y):\ A(x,y)\neq 0\}$ as in \eqref{vx} is formed by three segments of length $\sqrt3/3$ and forming equal angles of $\frac23\pi$ with each other, and thus the optimal shape in \eqref{subdopt1} for $\mathcal V=\mathcal V_x$ is an equilateral triangle of sidelength $2$. Thus $C_{\mathcal V_x}=C_{A,G}=\sqrt 3$. The set $H$ from \eqref{hlattice} is a regular hexagon of sidelength $2$, and thus of area $|H|=6\sqrt 3$. Therefore \eqref{isopineqall} gives the following {\bf isoperimetric inequality in the honeycomb graph} $G$
\[
\forall X\subset V,\quad (\sharp X)^2\le 6\ \sharp\overrightarrow{\partial X}.
\]
As a dilated copy of $H$ can be tessellated by tiles from $K$, we have that the above inequality is sharp and is achieved precisely by $X$ of cardinality $6k^2, k\in\mathbb N$, and whose dual cells in $K$ tile an equilateral hexagon.

\subsubsection{Deformations of the honeycomb graph example}
The combinatorial inequality obtained in the previous subsection begs the question of whether our method can be applied to other combinatorially equivalent rectilinear graphs with different side lenghts and weights. We verify that this is the case for any affine deformation of the corresponding graph $G$. This verification is relatively simplified by the fact that in this case we have only three different edge lengths $\ell_1,\ell_2,\ell_3$, corresponding to opposite sides of the honeycomb hexagonal cells. At each vertex of $G$ three edges with different edge lengths meet, and the reciprocal tessellation is by isometric triangles. The conditions of Proposition \ref{corstrong} are verified and the shape $H$ is an affine deformation of (but is combinatorially equivalent to) the one from the honeycomb graph, as expected. 

\medskip

Another option for creating new examples is to modify the weights $g, A$ for one of the above examples. The outcome of this operation is to change the ratios of sidelengths of the shape $H$, without modifying their orientations. The hypotheses of Proposition \ref{corstrong} are then still verified, but combinatorially the isoperimetric shapes will change.
\subsubsection{An interesting problematic case in the plane}
Isoperimetric shapes in graphs may exist and may have a form more complicated than \eqref{isopineqall}, in which case the framework of Proposition \ref{corstrong} does not help for their classification.

\medskip

This is the case of the triangular lattice graph, i.e. the $1$-skeleton of the tiling from Section \ref{trix}. In this case the optimal shape realizing the constant $C_{\mathcal V_x}$ at each $x$ in the graph, can be found to be again a regular hexagon, however if we try for example to find the shape in \eqref{hlattice}, it is once more a hexagon for this case. Since $H$ cannot be tessellated by hexagons, this is a first hint that an inequality of the form \ref{isopineqall} cannot hold in this graph.

\medskip

By using the duality between the triangular graph and the honeycomb graph from Section \ref{trix} and using basic topological invariants, we reach the conclusion that the regular hexagon shapes satisfy an isoperimetric-type inequality, see \eqref{mainthm}, whose proof is laid out in Section \ref{trigraph}. Investigating the structure of graph isoperimetric inequalities such as \eqref{mainthm} seems to be an interesting future research direction. As indicated by the proof in the simple case of the triangular graph, this may involve nontrivial new tools outside the scope of the present work.

\subsection{Examples based on tilings of $\mathbb R^3$ by tetrahedra or parallelotopes}

\subsubsection{Body-centered cubic lattice (BCC)}\label{secbcc}

We consider the so-called BCC lattice defined as follows (with the factor of $2$ included for simplicity of notation)
\[
\Lambda:= \{(1, 1, 1), (0,0,0)\}+2\mathbb Z^3.
\]
We consider the tessellation $K$ given as the Delone triangulation corresponding to this lattice. Recal that the Delone triangulation corresponding to a point configuration is the decomposition of space into simplices with vertices in the configuration such that the circunscribed spheres are not containing any points of the configuration in their interiors (see e.g. \cite{emp} for more details). As a rectilinear graph $G$ reciprocal to the tessellation $K$ is the $1$-skeleton of the Voronoi diagram corresponding to $\Lambda$, whose cells are all congruent to the truncated octahedron (also known as a permutahedron) with vertices 
\[
 \left\{\left(0,\pm\frac12,\pm 1\right)\text{ and permutations thereof}\right\}.
\]
We take $A(x,y)=g(x,y)=1$ for all sides of $G$, and we verify the hypotheses of Proposition \ref{corstrong} as follows:

\begin{itemize}
 \item The cells of $K$ are all cogruent to a so-called disphenoid tetrahedron, such as the one with vertices $(0,0,\pm 1), (1, \pm 1,0)$.
 \item As $A$ is constant and the sides of $G$ all have equal lengths, the optimizers in the definition of $C_{\mathcal V_x}$ all are equal disphenoid tetrahedra, and we have $C_{\mathcal V_x}=C_{G,A}=\frac1{12}$.
 \item The shape $H$ as in \eqref{hlattice} is given by a rhombic dodecahedron, which can be tessellated by $24k^3$ congruent disphenoid tetrahedra from rescalings of the tessellation $K$, for $k\in\mathbb N$. Furthermore, we have $|H|=2$.
\end{itemize}
Therefore, in this case \eqref{isopineqall} reads 
\[
 (\sharp X)^2 \le \frac{1}{24}(\sharp\overrightarrow{\partial X})^3, 
\]
and is achieved for $\sharp X=24k^3$, in case the cells in $K$ dual to vertices in $X$ tessellate a rhombic dodecahedron which has the same shape as $H$.

\subsubsection{Tessellations with vertex set equal to the FCC lattice}

As an indication of the existence of less symmetric examples than the above, we try to replicate the construction of Section \ref{secbcc} with the face-centered cubic (FCC) lattice defined as follows:
\[
 \Lambda:=\{(0,0,0),(1,1,0), (1,0,1), (0,1,1)\}+2\mathbb Z^3.
\]
The Delaunay tessellation of the above $\Lambda$ is usually given as a tessellation by regular tetrahedra and octahedra, of which the octahedra therefore have volume $4$ times the one of the tetrahedra. We thus cannot consider the reciprocal graph of this tessellation, which does not satisfy the hypotheses of Proposition \ref{corstrong}. Instead, we will divide each octahedron in $4$ equal tetrahedra along the coordinate planes with normals $e_1,e_2$ (note that all octahedra are translates of $\{(\pm1,\pm1,\pm1)\}$). The corresponding tessellation $K$ of $\mathbb R^3$ can be equivalently obtained as the complement of the plane arrangement
\begin{equation}\label{hyparrfcc}
 \{\pi_v(c):\ v\in\mathcal V, c\in \mathbb Z\},\quad \text{where}\quad\left\{
 \begin{array}{l}
 \mathcal V:=\left\{\left(\frac12,0,0\right), \left(0,\frac12,0\right), \left(\frac12,\pm\frac12,\pm\frac12\right)\right\},\\[3mm]
 \pi_v(c):=\{x\in\mathbb R^d:\ \langle v, x\rangle =c\}.
 \end{array}\right.
\end{equation}
To fix the reciprocal graph $G$ to the tessellation $K$ and the weights $A,g$ so that the hypotheses of Proposition \ref{corstrong} hold, note first that there are two types of cells in $K$:
\begin{itemize}
\item Type $1$ cells, congruent to regular tetrahedra of sidelength $\sqrt2$,\item Type $2$ cells, congruent to tetrahedra which are a quarter of an octahedron of sidelength $\sqrt2$.
\end{itemize}
 These cells have two types of facets: 
 \begin{itemize}
 \item Type $1$ facets: regular equilateral triangles with sidelength $\sqrt 2$,\item Type $2$ facets: isosceles right triangles with sidelengths $\sqrt 2, 2$.
 \end{itemize}
 Areas of facets of type $j$ are denoted $F_j$. Combinatorially, $G$ is the dual of $K$, and we require edges $[x,y]$ in $G$ to be normal directions of corresponding $2$-facets from $K$. Edges orthogonal to type $j$ facets have lengths denoted $\ell_j$. We will choose constants $A_j, g_j$ and set $A(x,y)=A_j, g(x,y)=g_j$ on edges of type $j$ for $j=1,2$. The optimal constants of the vertices dual to cells of type $j$ in $K$ will be denoted $C_{\mathcal V_j}$. Denoting $F_j:=\mathcal H^2(F_{x,y})$ for facets $F_{x,y}$ of type $j$ we have $F_1=\frac{\sqrt3}{2}, F_2=1$ and setting $A_2=g_2=\ell_2=1$, \eqref{conststuf} gives equations 
\[
 C_{\mathcal V_1}=C_{\mathcal V_2},\quad \ell_1 A_1^2=\frac{g_1^2}{\ell_1}=\frac{\sqrt3}{2},
\]
and we calculate $C_{\mathcal V_1}=C_{\mathcal V_2}=\frac{\sqrt3}{8}$, as the optimal shapes only depend on  $\ell_j A_j^2$ and on the angles between facets, all of which are already determined above. We then may fix $\ell_1,A_1,g_1$ satisfying the above arbitrarily, hich leaves one degree of freedom. 

\medskip

The shape $H$ as in \eqref{hlattice} is a convex set with facet normal vectors $\pm \mathcal V$ where $\mathcal V$ is as in \eqref{hyparrfcc}, and thus equals a regular octahedron truncated at two pairs of antipodal vertices, at equal distances from the center, where these distances depend on the ratio $A_1/g_1$. This figure can be tessellated by dilated cells of $K$ for special choices of $A_1/g_1$. We thus verify all conditions of Proposition \ref{corstrong} in this case.

\subsection{Higher-dimensional examples based on Coxeter triangulations} The examples in this section include the ones from the previous sections as special cases. A {\bf Coxeter triangulation} is a triangulation whose cell interiors are given by the complements of planes from a plane arrangement composed of families of hyperplanes perpendicular to vectors from a root system, whose cells are fundamental domains of an action by a finitely generated discrete subgroup $W$ of isometries of $\mathbb R^d$. For a classification and more properties of Coxeter triangulations, see \cite{coxtri} and \cite{humphreys}. 

\medskip

In this case we take $K$ to be the images of the fundamental cells under the action of $W$, $G$ to be the reciprocal graph of $K$ in which the edge $[x,y]$ reciprocal to a facet $F_{x,y}$ has lenght $\frac{1}{\mathcal H^{d-1}(F_{x,y})}$. Then we take $A(x,y)=\sqrt{\frac{\mathcal H^{d-1}(F_{x,y})}{|y-x|}}$. The shape $H$ from \ref{hlattice} is then tessellated by rescaled fundamental cells, and thus the hypotheses of Proposition \ref{corstrong} are satisfied.

\medskip

A complete classification of the case of graphs and tilings satisfying Proposition \ref{corstrong} coming from hyperplane arrangements seems within reach, but it goes beyond the scope of the present work and is left for future work.

\subsection{Building new examples from old ones}
\subsubsection{Orthogonal products}
Here we will discuss the posibility to form examples in the product space of two or more ambient spaces where there are cases of already known examples. 

\medskip

Suppose that for $j=1,2$ we have given tessellations $K_j$ of $\mathbb R^{d_j}$, rectilinear graphs $G_j$ in $\mathbb R^{d_j}$ dual to $K_j$ and positive symmetric weight functions $A_j, g_j$ such that the conditions of Proposition \ref{corstrong} hold. We denote the constant values of the discrete isoperimetric constants by $C_j:=C_{A_j,G_j}$ and the isoperimetric shapes by $H_j$, for $j=1,2$. 

\medskip

With the above conditions met, we let $d:=d_1+d_2$ and construct construct $K, G, A, g$ satisfying Proposition \ref{corstrong} in $\mathbb R^d$, as follows. 
\begin{itemize} \item The tessellation $K$ having cells $A_1\times A_2$, where $A_j$ is a cell of $K_j$ for $j=1,2$. 
 \item As a reciprocal to $K$ we may take the rectilinear graph $G$ with edges 
 \begin{equation}\label{segmgj}\{z_1\}\times[x_2,y_2]\quad\text{ and }\quad[x_1,y_1]\times\{z_2\},\quad \text{with $z_j$ vertices in $G_j$ and $[x_j,y_j]$ edges of $G_j$ for $j=1,2$.}
 \end{equation}

 \item The functions $A,g$ to take values $\sqrt{C_1} A_2(x_2,y_2), \sqrt{C_1} g_2(x_2,y_2)$ on edges from the left side of \eqref{segmgj} and value $\sqrt{C_2} A_1(x_1,y_1), \sqrt{C_1} g_1(x_1,y_1)$ on edges from the right side of \eqref{segmgj}.
\end{itemize}
With the above definitions we have $C_{\mathcal V_{(x_1,x_2)}}=\mathcal C_{V_{x_1}} \mathcal C_{V_{x_2}}=C_1C_2$ and the set $H$ defined in \eqref{hlattice} in this case is $H=H_1\times H_2$, which is tessellated by tiles in $K$ precisely if $H_j$ is tessellated by tiles in $K_j$ for $j=1,2$. Therefore the conditions of Proposition \ref{corstrong} hold and the isoperimetric constants and optimal shapes are obtained as a product of the ones with indices $j=1,2$.

\subsubsection{Cell subdivision}
If we have a Delaunay triangulation $K_0$ of $\mathbb R^d$ with equal volume simplicial cells and a reciprocal graph $G_0$, in which the conditions of Proposition \ref{corstrong} for obtaining \eqref{isopineqall} hold, then we may subdivide each cell into $d+1$ equal volume simplices by adding one new vertex in its interior. We then obtain a new triangulation $K$ and associate to it a reciprocal $G$. The subdivision $K$ is unique, and in $G$ each vertex of $G_0$ has been replaced by $d+1$ new vertices forming the $1$-skeleton of a simplex. The unit vectors of edges in $G$ are uniquely determined, and we may fix lengths of edges arbitrarily. We then have the freedom to fix $A(x,y), g(x,y)$ on each new edge, so that \eqref{conststuf} is still verified. 

\medskip

In this case the definiton \eqref{hlattice} defines a set which will in general {\it not} be tessellated by simplices of the new triangulation. However, dilations of the $H_g$ coming from the initial graph $G_0$ will be tessellable by simplices of $K$, and the proof that equality in \eqref{isopineqall} is achieved with this $H_g$ and for subsets of $G$ dual to these tessellations as for $G$, can be performed exactly like in Proposition \ref{corstrong}.

\subsubsection{New inequalities for associated graphs, via precise resummation using topological invariants: case study for the triangular graph}\label{trigraph}

The triangular graph $G$ is the graph with vertex set equal to the triangular lattice $\Lambda$ defined in Section \ref{trix}, and edges connecting nearest-neighbors in $\Lambda$. For simplicity, we choose $g\equiv 1, A\equiv 1$, although extension to the general case is possible and straightforward. 

\medskip

Note that in Proposition \ref{corstrong} the dual cells in $K$ are hexagons with edges orthogonal to the directions of nearest-neighbor vectors in $\Lambda$ and so is the set $H_g$ from \eqref{hlattice}. Therefore only for the trivial case $\sharp X=1$ could a dilated copy of $H_g$ be tessellated by cells of $K$. On the other hand, optimal isoperimetric sets at fixed $\sharp X$ can be found, and they approximate regular hexagons, being equal to hexagons for special values of $\sharp X$. 

\medskip

We show below (as is not surprising, see e.g. \cite{heitmannradin}, of which the final result found here is a consequence, or \cite{delfries} which also uses Gauss-Bonnet methods) that a sharp isoperimetric inequality which is better adapted to the problem can be found, in which all optimal hexagons satisfy the equality. The precise inequality can be found under the ansatz that $p(\sharp X) = q(\sharp \overrightarrow{\partial X})$ holds for each $X$ forming a regular hexagon configuration in $\Lambda$ and $p,q$ are polynomials of degrees $2,3$ respectively. However we find the same formula via a direct geometric approach, which we hope that can give a first example for a general class of constructions to be studied in future work. See the titles of the below paragraphs for the structure the reasoning.

\medskip 

{\bf Connectedness of optimizers.} We suppose that $\Omega$ is an isoperimetric set, for which $(\sharp\partial\Omega)^2/\sharp\Omega$ is minimal. Firstly, the below simple result implies that $\oo$ must be connected. The lemma is well known but we include a proof for completeness.

\begin{lemma}\label{lem:minconnected}
 Let $V$ be a lattice and assume that $B$ is such that $\pm B$ contains $0$ in the interior of its convex hull. Assume that $\oo\subset V$ is not connected and it has connected components $\oo_1,\ldots,\oo_n$ with interiors $\Omega_1,\ldots,\Omega_n$. For $k=1,\ldots,n$ there exist $v_k\in V$ such that the $\Omega_k+v_k$ are pairwise disjoint and $\Omega':=\bigcup_{k=1}^n(\Omega_k+v_k)$ is connected. 
 For any such $\Omega'$ there holds
 \begin{equation}\label{betterconnected}
  \sharp\Omega=\sharp\Omega',\quad \sharp \partial \Omega>\sharp \Omega'.
 \end{equation}
\end{lemma}
\begin{proof}
We show the result by induction on $n\ge 2$. We first present the case of $\Omega$ having two connected components, $n=2$. Let $w_1$ be an extreme point of $\Omega_1$ in direction $b$ and $w_2$ be an extreme point of $\Omega_2$ in direction $-b$. Then there exists a direction $b'\in \pm B$ such that $\langle b,b'\rangle>0$, $w_1$ has no neighbor in $\Omega_1$ in the direction $b'$ and $w_2$ has no neighbor in $\Omega_2$ in direction $-b'$. The existence of such $b'$ follows from the assumption that $\pm B$ contains $0$ in the interior of its convex hull. We then find that if to $\Omega_1,\Omega_2$ we apply respectively translation vectors $v_1=0$ and $v_2=w_1-w_2+b'$, then the images of $w_1, w_2$ become neighbors, thus $\Omega'$ is connected. All the other properties required in the lemma can be easily verified. This concludes the discussion for $n=2$. For $n>2$ we apply the above procedure to $\bigcup_{k=1}^{n-1}\Omega_k$ and $\Omega_n$ instead, and we obtain translation $v_n$ that can be applied to $\Omega_n$ so that it becomes ``attached'' via an extreme point to at least one of the $\Omega_k,k=1,\ldots,n-1$. This produces a new set $\Omega'$ which has strictly less connected components, to which we can apply the inductive hypothesis, allowing to conclude.
\end{proof}

{\bf Passing information between the dual and primal graph.}
We assume that $\Omega$ is a union of triangles (for the definition see Section \label{sec:dugra}), that the vertices $\Omega$ with less than $6$ neighbors in $\Omega$ form a single cycle without repeated vertices, and that $\Omega^*$ is the graph corresponding to the faces of $\Omega$. We then use the following notations:
\begin{itemize}
 \item $X:=\sharp \Omega$, $Y=\sharp\overrightarrow{\partial\Omega}$, $X^*:=\sharp \Omega^*$ and $Y^*:=\sharp\overrightarrow{\partial\Omega^*}$.
 \item For $i\in\{1,\ldots,6\}$, $V_i$ are the vertices in $\Omega$ which have $i$ neighbors in $\Omega$, and $a_i:=\sharp V_i$.
\end{itemize}
Then note the following properties connecting the $a_i$ and $X,Y,X^*,Y^*$:
\begin{enumerate}
 \item Since $\Omega$ is a union of triangles, we have $a_5=0$.Then 
 \[
 X=a_1+a_2+a_3+a_4+a_6.
 \]
 \item Since the vertices in $C:=V_1\cup V_2\cup V_3\cup V_4$ form a single cycle without repeated vertices, it means that vertices from $V_i$ have $i+1$ neighbors in $\Omega$ for $i\in\{1,\ldots,4\}$ and vertices from $V_6$ have $6$ neighbors in $\Omega$.
 \item For $x\in \Omega$ the number of boundary edges it belongs to plus the number of neighbors of $x$ equal $6$.
 \item Due to the previous two points, 
 \[
 Y=4a_1+3a_2+2a_3+a_4.
 \]
 \item By the Gauss-Bonnet theorem, 
 \begin{equation}\label{gb}
 2a_1+a_2-a_4=6.
 \end{equation}
 \item As $C$ is a cycle without repeating vertices, it has equal number of vertices and  edges connecting them, and each such edge corresponds to exactly one triangle to the exterior of $\Omega$, thus to an element of $\overrightarrow{\partial\Omega^*}$. So we have:
 \begin{equation}\label{ystar}
 Y^*=a_1+a_2+a_3+a_4=X-a_6.
 \end{equation}
 \item Each half-edge in $E_\Omega$ belongs to one vertex. Thus point 3 gives
 \begin{equation}\label{sharplomega}
  2\sharp E_\Omega=6a_6+5a_4+4a_3+3a_2+2a_1=6X-Y.
 \end{equation}
 \item Consider the polygonal complex with vertices $\Omega$, edges $E_\Omega$ and faces given by the triangles contained in $\Omega$. By assumption this complex has only one boundary component, thus it has Euler characteristic $1$. Its faces are in bijection with $\Omega^*$, thus, using also \eqref{sharplomega},
 \begin{equation}\label{xstar}
 X+X^*-\sharp E_\Omega=1\quad\Leftrightarrow\quad 2=2X+2X^*-6X+Y=2X^*-4X+Y.
 \end{equation}
\end{enumerate}
We rewrite all the equations in terms of $X,Y,X^*,Y^*$ only. \eqref{gb} as $Y-2X=6-2a_6$, from which we get via \eqref{ystar}
\begin{equation}\label{ystar2}
 Y-2X=6+2Y^*-2X\quad\Leftrightarrow\quad Y^*=\frac{Y}{2}-3.
\end{equation}
Reordering \eqref{xstar} we find
\begin{equation}\label{xstar2}
 X^*=2X-\frac{Y}{2}+1.
\end{equation}
then the result of Section \ref{trix} states that 
\begin{equation}\label{mainthm}
\frac{(\sharp\overrightarrow{\partial\Omega^*})^2}{\sharp\Omega^*}\ge 6\quad\Leftrightarrow\quad\frac{(\sharp\overrightarrow{\partial\Omega}-6)^2}{4\sharp\Omega - \sharp\overrightarrow{\partial\Omega} +2}\ge 12.
\end{equation}
It is direct to verify that equality is achieved for $\Omega$ equal to a ``perfect hexagon'' $\Omega:=H_k$ in the triangular grid, which notation indicates as $k$ the number of graph edges forming each side of the hexagon. For $k\ge 0$ we find that
\begin{itemize}
 \item $\sharp H_k=3k^2+3k+1$
 \item $\sharp\partial H_k=12k+6$
\end{itemize}

\section{The upper bound in the discrete Neumann boundary problem}\label{sec3}

In this section we focus on a possible venue for extending the class of examples reacheable through the setup of Section \ref{sec5}. We aim to improve the understanding of the optimal constant upper bound in the Neumann boundary value problem \eqref{ubdp}, which enters both in the definition of the optimal isoperimetric constant and in the understanding of the corresponding isoperimetric shape.

\medskip

We will formulate the definition of the optimal value of this constant and prove existence and an alternative characterization \mipc{add theorem}.

\subsection{Lowest upper bounds on the discrete Laplacian under fixed boundary value condition}
For any function $g:\bo\to \mathbb R$ denote
\begin{equation}\label{cg}
c_g:=\sum_{y\in\bo}\sharp\{x\in \Omega:\ A(x,y)\neq 0\} g(y).
\end{equation}
We find that $f(x):=c_g/\sharp\Omega$ for all $x\in\Omega$ together with the above $g$ satisfies condition \eqref{necneum11} from Lemma \ref{lemexistsolineq}, and thus there exists a solution $u$ to
\begin{equation}\label{bdval}
\displaystyle  \left\{\begin{array}{rl}
\Delta_A u \le \frac{c_g}{\sharp\Omega}  &  \text{in } \Omega,\\
\frac{\partial u}{\partial \nu_\Omega} = \frac{g}{A} &  \text{on } \overrightarrow{\partial \Omega},
\end{array}\right. 
\end{equation}
with boundary condition interpreted as in \eqref{newbdcond}. The lowest possible feasible upper bound in \eqref{bdval} is
\begin{equation}\label{minmax}
 C(g,\Omega):=\inf\left\{\max_{x\in\Omega}\Delta_A u(x):\ u:\oo\to\mathbb R,\text{ with } \frac{\partial u}{\partial\nu_\Omega}=\frac{g}{A}\text{ in the sense of \eqref{newbdcond}}\right\}.
\end{equation}
Due to Lemma \ref{newbdcond}, $c_g/\sharp\Omega$ is a competitor for the minimization \eqref{minmax} and thus we have the following:
\begin{lemma} With the above notations \eqref{cg} and \eqref{minmax}, there holds
\begin{equation}\label{startbound}
C(g,\Omega)\le\frac{c_g}{\sharp\Omega}.
\end{equation}
\end{lemma}
\subsection{Study of the optimal $u$ in \eqref{minmax} and lower bounds for $C(g,\Omega)$.} We now show that the infimum in \eqref{minmax} is achieved and find a lower bound for $C(g,\Omega)$. As a consequence of $\partial u/\partial\nu_\Omega=g$ being interpreted as \eqref{newbdcond}, there exists a set $\vec E'\subset \pomega$ such that $\bo=\{y:\ (x,y)\in \vec E'\}$ and such that $g(y)=u(y)-u(x)$ for each $(x,y)\in \vec E'$. Thus we have the more explicit reformulation
\begin{subequations}\label{minmaxlarge}\begin{equation}
 C(g,\Omega)=\min_{E'\in\mathcal E'}C(g,\Omega,\vec E'):=\min_{E'\in\mathcal E'}\inf\left\{\max_{x\in\Omega}\Delta_A u(x):\ u:\oo\to\mathbb R,\ (\forall (x,y)\in \vec E'), u(y)-u(x)=\frac{g(y)}{A(x,y)}\right\},
\end{equation}
where 
\begin{equation}
 \mathcal E':=\{\vec E'\subset \pomega:\ \bo=\{y:\ (x,y)\in \vec E' \},\ \sharp \vec E'=\sharp(\bo)\}.
\end{equation}
\end{subequations}
We next fix $g$ and $\vec E'\in\mathcal E'$ and consider the problem $C(g,\Omega,\vec E')$ by considering first the dual problem. This requires a couple of preliminary definitions.

\medskip

\par{\bf Directed Laplacian.} A directed graph $\vec G=(V,\vec E)$ is defined by a set of vertices $V$ and a set of edges $\vec E\subset V\times V$. We assume that $\vec G$ is such that for each $x\in V$ the set $\{y\in V:\ (x,y)\in \vec E\}$ is nonempty, i.e. that $\vec G$ has no sinks. The directed Laplacian $\vec\Delta$ of the directed graph $\vec G$ is an operator sending the set of functions $u:V\to \R$ to itself, defined by
\[
 \vec \Delta u(x):=\sum_{y: (x,y)\in\vec E}(u(y)-u(x)).
\]
If furthermore a weight $A(x,y)$ is present on edges $(x,y)\in \vec E$, then we can define $\vec \Delta_A u(x):=\sum_{y: (x,y)\in\vec E}A^2(x,y)(u(y)-u(x))$. 

\medskip
\par{\bf As the modifications for general symmetric weights $A(x,y)$ are straightforward, from now on we restrict to the case that $A(x,y)\equiv 1$ on $\vec E$, and leave the extensions to the general case to the reader.} We thus denote $\vec \Delta_A=\vec \Delta$. In order to avoid confusion, we mention here that below we will also use the notation $\vec \Delta_{\vec E'}$ when it is important to keep in mind the set of edges $\vec E'$ as in \eqref{minmaxlarge}.

\medskip

\par{\bf $C(g,\Omega,\vec E')$ reexpressed in terms of directed Laplacian.} For $\vec E'\in\mathcal E'$, the operator $A_{\vec E'}$ sending $\R^{\oo}$ to itself and defined by
\[
 (A_{\vec E'}u)(x)=\left\{\begin{array}{ll} \Delta u(x)&\mbox{ if }x\in\Omega,\\[3mm]
 u(y)-u(x)&\mbox{ if } x\in\oo\setminus\Omega, (x,y)\in\vec E',\end{array}\right.
\]
is the same as the directed Laplacian $\vec \Delta_{\vec E'}$ on the directed graph $\vec G_{\vec E'}$ with vertices $\oo$ and edges 
\[
\{(x,y):\ x\in\Omega, (x,y)\in \vec E_{\oo}\}\cup \{(y,x):\ (x,y)\in\vec E'\}.
\]
\par{\bf Properties of $A_{\vec E'}$.} We note that like for the Laplacian, the directed Laplacian over a graph can be identically zero only on functions that are constant on each connected component of the graph, where now for testing connectedness only oriented paths are allowed. Indeed, if $\vec \Delta u\equiv 0$ at all vertices, then at each vertex $x_0\in G$, this says that $u(x_0)$ is equal to the average of $u$ over all $y$ such that $(x_0,y)\in \vec E'$, no $u$ cannot have strict local minima or maxima.

\par {\bf From now on we assume $\Omega$ connected,} for simplicity, as in the case of $\Omega$ with multiple components we can reason componentwise as in  In this case, we find that $\mathrm{Ker}(A_{\vec E'})$ is formed by just the constant functions $u$ over $\oo$, thus is $1$-dimensional. The operator $A_{\vec E'}$ is in general not self-adjoint, but we can nevertheless gain the information that the image of $A_{\vec E'}$ has codimension $1$. 

\par {\bf Compatibility condition and the image of $A_{\vec E'}$.} We now characterize the orthogonal to the image of $A_{\vec E'}$ with respect to the usual inner product $\langle u,u'\rangle=\sum_{x\in\oo} u(x) u'(x)$. We know from the previous paragraph that the space of functions $v:\oo\to\R$ orthogonal complement of $\mathrm{Im}A_{\vec E'}$ has dimension $1$, thus contains a nonzero function, unique up to scalar multiples. The condition determining this function is $\langle \vec \Delta_{\vec E'}u, v\rangle =0$ for all $u$. This can be explicitly written (note that the computation below holds on any connected directed graph):
\begin{eqnarray}\label{veqn}
 \lefteqn{(\forall u),\ 0=\sum_{x\in\oo}v(x)\vec \Delta_{\vec E'} u(x)=\sum_{x\in\oo}v(x)\sum_{y:x\to y}(u(y)-u(x))}\nonumber\\
 &=&\sum_{x\in\oo}v(x)\left(-d_x^{out} u(x) +\sum_{y:x\to y}u(y)\right)=\sum_{x\in\oo} u(x)\left(- d_x^{out} v(x)+ \sum_{y:y\to x}v(y) \right).
\end{eqnarray}
In the above $d_x^{out}$ (usually called the ``outgoing degree at $x$'') is the number of edges of the graph starting at $x$. Since \eqref{veqn} must hold for all choices of $u$, we find that $v$ solves at each vertex the equation
\begin{equation}\label{veqn1}
 d_x^{out}v(x) =\sum_{y:y\to x}v(y).
\end{equation}
This equation seems to not have been studied in the literature, however in our case we know that a solution exists because we control the kernel of $A_{\vec E'}$.

\par{\bf We will reserve the notation ``$v$'', in the remainder of this section, for the solution of \eqref{veqn1} over $\oo$.}

We say that two vertices $x,y$ in a directed graph are \emph{connected} if there exists an oriented path from $x$ to $y$ or from $y$ to $x$. We say that a directed graph is connected if all pairs of vertices are connected.
\begin{proposition}\label{vnochange}
 Assume that $\vec E'$ and $v$ are as in \eqref{minmaxlarge}, in particular we study the directed Laplacian on $\vec G$, a connected directed graph. Then $v$ does not vanish and as constant sign on $\Omega$.
\end{proposition}
\begin{proof}
We will assume that $v(x_0)=0$ for some $x_0\in \Omega$, and reach a contradiction from this assumption. We first note that if $v(x_0)=0$ then there exists a solution $u$ to the system 
\begin{equation}\label{diraclap}
 \left\{
 \begin{array}{ll}
  \vec \Delta_{E'} u(x_0)=1,&\\[3mm]
  \vec \Delta_{E'} u(x)=0,&\mbox{ if }x\in \Omega\setminus\{x_0\},\\[3mm]
  u(y)-u(x)=0,&\mbox{ if }(x,y)\in\vec E'.
 \end{array}
\right.
\end{equation}
Let $\vec G'$ be a connected component of $\vec G\setminus\{x_0\}$. Then $\vec G'$ is connected to $x_0$ in $\vec G$ and is a connected directed graph.  We note that the condition $\vec \Delta_{E'}u(x)=0$ on $\vec G'$ implies that $u$ is constant on $\vec G'$, as it cannot achieve local maxima or minima. By our initial hypotheses on $\vec G$ and $u$, we thus have that $u$ is constant on $\vec G'$. Since $\vec G'\cup\{x_0\}$ forms a connected subgraph of $\vec G$, there exists $x\in \vec G'$ such that $(x,x_0)$ is an edge in $\vec G$. We have $u(x)=0$ and since $u$ solves the discrete PDE above, we also have that $u(x)$ is the average of the values of $u$ on the neighbors of $x$:
\[ 
 u(x)=\frac{1}{\sharp\{y:\ \exists x\to y\}}\left(\sum_{\substack{y:x\to y\\ y\neq x_0}}u(y) + u(x_0)\right).
\]
But $u(y)=u(x)$ for all $y\neq x_0$ such that $x\to y$, and then the above equation also implies that $u(x_0)$ is also equal to $u(x)$. Thus, $u$ is constant on $\vec G'\cup\{x_0\}$. Repeating this for all connected components of $\vec G\setminus \{x_0\}$, we find that $u$ is constant. On the other hand, we have that by our PDE, $u(x_0)$ is strictly higher than the average of the values of $u$ on the neighbors of $x_0$. This gives the desired contradiction, and shows that \eqref{diraclap} does not have a solution for $x_0\in\Omega$, and thus that $v$ is nonzero over $\Omega$, as desired.

\medskip

A similar reasoning shows that $v$ does not change sign. If $v$ changes sign then we can find $u$ with $\Delta u(x_0)=a, \Delta u(y_0)=b$ with $a,b>0$ and we get a contradiction. Needs to be done.
\end{proof}

As a corollary, we find that $C(g,\Omega,\vec E')>-\infty$:
\begin{corollary}\label{cor:lbcgomega}
 If $\Omega$ is connected and $v$ is defined as above, normalized so that $v>0$ over $\Omega$, then we have the bound
 \begin{equation}\label{lowerbound}
  C(g,\Omega,\vec E')\ge \min\left\{0,-\frac{\sum_{y\in\oo\setminus\Omega}g(y)v(y)}{\sharp\Omega\min_{x\in\Omega}v(x)}\right\}>-\infty.
 \end{equation}
 In particular, the infimum in the definition of $C(g,\Omega,\vec E')$ is achieved and is thus a minimum.
 \end{corollary}
\begin{proof}
The last inequality follows directly due to the fact that $g,v$ take finite values and $\Omega$ is a finite set, together with Proposition \ref{vnochange}, which ensures that the minimum of $v$ over $\Omega$ is $>0$. If \eqref{lowerbound} were false, we would have a solution $u$ with right hand side $f$ such that $f<0$ and $f<-\frac{\sum_{y\in\oo\setminus\Omega}g(y)v(y)}{\sharp\Omega\min_{x\in\Omega}v(x)}$ over $\Omega$. But then the compatibility condition  gives
\begin{eqnarray*}
 -\sum_{y\in\oo\setminus\Omega}g(y)v(y)&=&\sum_{x\in\Omega}f(x)v(x)\\
 &<& \min\left\{0,-\ \frac{\sum_{y\in\oo\setminus\Omega}g(y)v(y)}{\sharp\Omega\min_{x\in\Omega}v(x)}\right\}\sum_{x\in\Omega}v(x)\\
 &\le&\min\left\{0,-\ \frac{\sum_{y\in\oo\setminus\Omega}g(y)v(y)}{\sharp\Omega\min_{x\in\Omega}v(x)}\right\}\sharp\Omega\min_{x\in\Omega}v(x)\\
 &=&\min\left\{0,- \sum_{y\in\oo\setminus\Omega}g(y)v(y)\right\},
\end{eqnarray*}
which gives a contradiction.
\end{proof}

{\bf Sign convention:} We assume that $v>0$, which has the only role of simplifying notations below. Recall that we are interested to get bounds on $C(g,\Omega,\vec E')$ in terms of $g$. These will be based on bounds on $M(\vec E'):=\frac{\max v}{\min v}$, which $M(\vec E')$ does not change if $v$ is replaced by $\lambda v$ with $\lambda\in \R$.

\par{\bf Condition for optimality in \eqref{minmaxlarge}.} 
 Note that the problem $C(g,\Omega,\vec E')$ can be cast as a linear programming problem, as follows:
\begin{equation}
 \min\left\{\max_{x\in\Omega}\Delta u(x):\ (\forall (x,y)\in\vec E'), u(y)-u(x)=g(y)\right\}=\min\left\{z:\ B\vec u = \vec g, \ (\forall x\in \Omega) z\ge A_x\cdot \vec u\right\},\label{lp}
\end{equation}
where we represent $u$ as a vector $\vec u\in \R^{\oo}$, and the variables of the problem are then $\vec u, z$ with $z$ one-dimensional, and we defined $B$ to be the matrix which applied to $\vec u$ gives the vector $(u(y)-u(x))_{y\in \oo\setminus \Omega}$ and $A_x\cdot\vec u=\Delta u(x)$. The equality \eqref{lp} is a simple exercise familiar to linear programming students.

\par We see that $\bar u, \bar z$ form an optimizer in \eqref{lp} if and only if 
\begin{enumerate}
\item We have $B\bar u =\vec g$ and $(\forall x\in \Omega) \bar z\ge A_x\cdot \bar u$.
\item If $(\tilde u, \tilde z)$ is a perturbation respecting the constraint, i.e. satisfying $B\tilde u=0$, and such that $\tilde z<0$, then it violates some of the inequalities which were saturated at $\bar u, \bar z$, i.e. 
\[\tilde z\ge A_x\cdot \tilde u\mbox{ for some }x\in\Omega\mbox{ at which }\bar z = A_x\cdot\bar u.
\]
Indeed, were this not the case, a small perturbation along $\tilde u, \tilde z$ would decrease the value of the $z$ variable while respecting the constraints, contradicting the optimality of $\bar u, \bar z$.
\end{enumerate}
Note that the variable $\tilde z$ can be eliminated in the point (2) above, and the resulting optimality condition can be then expressed in terms of our discrete equation as follows:
\begin{lemma}\label{condoptlem}
 $\bar u$ is an optimizer for $C(g,\Omega,\vec E')$ if and only if the following is true:
 \begin{equation}\label{condopt}
  \mbox{For all } \tilde u\mbox{ with $\vec E'$-boundary datum zero},\quad (\exists x\in\mathrm{argmax}_\Omega\Delta\bar u)\mbox{ at which }\Delta \tilde u(x)\ge 0. 
 \end{equation}
\end{lemma}
We are now ready for the following proposition.
\begin{proposition}\label{optconst}
Assume that $\Omega$ is connected and the minimum in $C(g,\Omega,\vec E')$ is achieved. Let $v$ be a nonzero solution of \eqref{veqn1} corresponding to $\vec E'$. Then if $\bar u$ is a function which achieves the minimum in \eqref{lp}, then there holds 
\begin{equation}\label{argmaxinclusion}
\mathrm{argmax}_{x\in\Omega}\Delta \bar u\supset \{x:\ v(x)\neq 0\}.
\end{equation}
\end{proposition}
\begin{proof}
 Assume that $\Delta \bar u$ is not constant, say that $x_0\in\Omega$ is not a maximum point of $\Delta \bar u$. Recall that $v$ being a nonzero solution to \eqref{veqn1}, it gives the direction perpendicular to the image of $A_{\vec E'}$. This means that a solution to 
 \[
    \left\{\begin{array}{ll}
          \Delta\tilde u(x)=f(x),&\mbox{ for }x\in \Omega,\\[3mm]
          \tilde u(y)-\tilde u(x)=0,&\mbox{ for }(x,y)\in\vec E'
          \end{array}
\right.
 \]
exists if and only if there holds
\begin{equation}\label{condexist}
 \sum_{x\in\Omega} v(x) f(x)=0.
\end{equation}
If by contradiction we assume that \eqref{argmaxinclusion} is false, this means that there is a point $x_0\in\Omega$ at which $\Delta\bar u$ is not maximal, and such that $v(x_0)\neq 0$. For such $x_0$, we can find a solution $\tilde u$ to the following equation
  \[
   \left\{\begin{array}{ll}
          \Delta\tilde u(x_0)=\frac{1}{v(x_0)}\sum_{x\in\Omega\setminus\{x_0\}}v(x),&\\[3mm]
          \Delta\tilde u(x)=-1,&\mbox{ for } x\in\Omega\setminus\{x_0\},\\[3mm]
          \tilde u(y)-\tilde u(x)=0,&\mbox{ for }(x,y)\in\vec E'.
          \end{array}
\right.
  \]
We defined the above right hand side so that \eqref{condexist} is verified, which as seen above guarantees that a solution $\tilde u$ exists. This solution contradicts \eqref{condopt} because $\Delta \tilde u<0$ over $\mathrm{argmax}_{x\in\Omega}\Delta \bar u$. This contradiction shows that \eqref{argmaxinclusion} holds.
\end{proof}
As a direct consequence of the Propositions \ref{vnochange} and \ref{optconst}, we have the following:
\begin{corollary}\label{lapconstant}
 If $\Omega$ is connected and the minimum in $C(g,\Omega, \vec E')$ is achieved by a function $\bar u$, then $\Delta \bar u$ is constant over $\Omega$.
\end{corollary}
As a consequence of the corollary, we have the following inequality, which is the content of Theorem \ref{thmconstnpintro}

\begin{equation}\label{cgomegal2}
 C(g,\Omega)=\min\left\{\frac{\sum_{y\in\bo}g(y)v_{\vec E'}(y)}{\sum_{x\in\Omega}v_{\vec E'}(x)}:\ v_{\vec E'}\mbox{ satisfies }\eqref{veqn1}\mbox{ on }\vec G_{\vec E'}\mbox{ for some }\vec E'\in\mathcal E'\right\}.
\end{equation}

\end{document}